\documentclass[a4paper,
9pt,
1p,oneside,onecolumn]{extarticle}
\usepackage{amsmath,amsthm,amssymb}
\usepackage{mathabx}
\usepackage[english]{babel}
\usepackage[ansinew]{inputenc}
\usepackage{verbatim}
\usepackage{calligra}

\usepackage{soul}

\usepackage{graphicx}
\usepackage{algorithm,algorithmic}
\usepackage{footmisc}
\usepackage{framed}
\usepackage{rotating}
\usepackage{wrapfig}
\usepackage[all]{xy}

\usepackage{setspace}

\usepackage{siunitx}

\usepackage{comment}
\usepackage{xcolor}

\newcommand{\xc}{\color{black}} 

\usepackage{hyperref}
\hypersetup{
colorlinks  = true,
citecolor = black, 
linkcolor = black
}

\DeclareMathOperator*{\argmin}{arg\,min}

\numberwithin{equation}{section}

\newtheorem{corollary}{Corollary}
\newtheorem{assumption}{Assumption}
\newtheorem{lemma}{Lemma}
\newtheorem{definition}{Definition}
\newtheorem{remark}{Remark}
\newtheorem{theorem}{Theorem}

\newcommand{\fexp}{\phi}  
\newcommand{\genpol}{v}  
\newcommand{\genfun}{f}  
\newcommand{\y}{y}  
\newcommand{\basisfuncexp}{\psi}  
\newcommand{\Rset}{\mathbb{R}}    
\newcommand{\Nset}{\mathbb{N}}    
\newcommand{\s}{d}  
\newcommand{\inddim}{q}
\newcommand{\indbasis}{j}

\newcommand{\indmeas}{i}
\newcommand{\M}{m}

\newcommand{\Pol}{\mathbb{P}}

\newcommand{\cdelta}{\xi}
\newcommand{\quadweights}{\alpha}

\newcommand{\param}{r}
\newcommand{\density}{\rho}

\newcommand{\gramian}{G}
\newcommand{\coef}{\beta}
\newcommand{\design}{D}
\newcommand{\rhs}{b}
\newcommand{\rhsfun}{\Phi}

\newcommand{\cI}{\mathbb{I}}
\newcommand{\parajac}{\theta}
\newcommand{\err}{g}
\newcommand{\matber}{Q}
\newcommand{\matsin}{X}
\newcommand{\matsum}{Z}
\newcommand{\Q}{I}

\newcommand{\tcheck}[1]{\textcolor{black}{#1}} 
\newcommand{\tc}[1]{\textcolor{black}{#1}} 

\newcommand{\E}{\mathbb{E}}

\begin{document}
\title{Stable high-order randomized cubature formulae in arbitrary dimension}
\author{Giovanni Migliorati\footnote{
Laboratoire Jacques-Louis Lions, 
Sorbonne Universit\'e,
Paris 75005, France. 
email: migliorati@ljll.math.upmc.fr} \footnote{corresponding author}
\and 
Fabio Nobile\footnote{MATHI-CSQI, \'Ecole Polytechnique F\'ed\'erale de Lausanne, Lausanne CH-1015, Switzerland. email: fabio.nobile@epfl.ch}
}
\date{\today}
\maketitle

\begin{abstract}
\noindent
We propose and analyse randomized cubature formulae for the numerical integration of functions with respect to a given probability measure $\mu$ defined on a domain $\Gamma \subseteq \mathbb{R}^d$, in any dimension $d$. Each cubature formula is exact on a given finite-dimensional subspace $V_n\subset L^2(\Gamma,\mu)$ of dimension $n$, and uses pointwise evaluations of the integrand function $\phi : \Gamma \to \mathbb{R}$ at $m>n$ independent random points. These points are drawn from a suitable auxiliary probability measure that depends on $V_n$. We show that, up to a logarithmic factor, a linear proportionality between $m$ and $n$ with dimension-independent constant ensures stability of the cubature formula with high probability. We also prove error estimates in probability and in expectation  
for any $n\geq 1$ and $m>n$, thus covering both preasymptotic and asymptotic regimes. 
Our analysis shows that the expected cubature error decays as $\sqrt{n/m}$ times the $L^2(\Gamma, \mu)$-best approximation error of $\phi$ in $V_n$. On the one hand, for fixed $n$ and $m\to \infty$ our cubature formula can be seen as a variance reduction technique for a Monte Carlo estimator, and can lead to enormous variance reduction for smooth integrand functions and subspaces $V_n$ with spectral approximation properties. On the other hand, when $n,m\to\infty$, our cubature becomes of high order with spectral convergence. 
As a further contribution, we analyse also another cubature whose expected error decays as $\sqrt{1/m}$ times the $L^2(\Gamma, \mu)$-best approximation error of $\phi$ in $V_n$, 
which is asymptotically optimal but with constants that can be larger in the preasymptotic regime. 
Finally we show that, under a more demanding (at least quadratic) proportionality betweeen $m$ and $n$, all the weights of the cubature are strictly positive with high probability. As an example of application, we discuss the case where the domain $\Gamma$ has the structure of Cartesian product, $\mu$ is a product measure on $\Gamma$ and $V_n$ contains algebraic multivariate polynomials. 
\end{abstract}

\noindent 
{\bf AMS classification numbers:}   
41A25, 
41A65, 
65D32. 



\section{Introduction}
\noindent
Let $\Gamma \subseteq \Rset^\s$ be a Borel set, $\mu$ be a Borel probability measure on $\Gamma$ absolutely continuous with respect to the Lebesgue measure $\lambda$, and denote with $\density:=d\mu/d\lambda :\Gamma\to \Rset$ its probability density.  
Given a function $\fexp:\Gamma \to \Rset$ in some smoothness class, we consider the problem of integrating $\fexp$ with respect to $\mu$ over $\Gamma$: 
\begin{equation}
I(\fexp):=\int_\Gamma \fexp(\y) d\mu(\y)=\int_\Gamma \fexp(\y) \density(\y) d\lambda(\y).  
\label{eq:integral_analytic}
\end{equation}

When the expression of $\fexp$ or the geometric shape of the domain $\Gamma$ are complicated, the exact calculation of $I(\fexp)$ might be too difficult, or not be possible at all, for example if the function $\fexp$ is not available in explicit form but can only be evaluated at any point $\y \in \Gamma$ at a certain (possibly high) cost, so that the number of evaluations should be limited as much as possible. Hence one resorts to the numerical approximation of the integral \eqref{eq:integral_analytic}, see \emph{e.g.}~\cite{DR84,Stroud1971}, that is known as the problem of numerical quadrature when $\s=1$ or numerical cubature when $\s\geq 2$, and that can become a challenging task as $\s$ increases due to the curse of dimensionality. In any dimension $\s\geq 1$ and given an integer $\M\geq 1$, we consider the $\M$-point quadrature/cubature formula 
\begin{equation}
\label{eq:def_quadrature_general}
\Q_{\M}(\fexp):= \sum_{\indmeas = 1}^{\M} \quadweights_\indmeas \fexp(\y_\indmeas), 
\end{equation}
where $\y_1,\ldots,\y_\M \in \Gamma $ are the nodes and $\quadweights_1,\ldots,\quadweights_\M \in \Rset$ are the weights.  
The nodes and weights should be chosen such that  
\begin{equation}
\Q_{\M}(\fexp) \approx I(\fexp).
\label{eq:approx_quad_int}
\end{equation} 

One approach to develop quadrature/cubature formulae imposes that \eqref{eq:def_quadrature_general} be exact on some given finite-dimensional linear function space $V_n$ over $\Gamma$, where $n:=\textrm{dim}(V_n)$.  
In principle one would like to have a formula that exactly integrates any function in $V_n$, \emph{i.e.}
$$
\Q_{\M}(\genpol) = I(\genpol), \quad \forall \, \genpol  \in V_n.  
$$
When $\s=1$ and $V_n$ is a polynomial space, the existence of such quadrature formulae has been first discussed in \cite{Christoffel} with general $\density$, extending earlier results in \cite{Gauss} with $\density \equiv 1$. 
An example of quadrature is the Gauss-Hermite quadrature formula, that exactly integrates any univariate algebraic polynomial up to degree $2\M-1$ using $\M$ points, where integration is intended with respect to the Gaussian probability measure on $\mathbb{R}$. 

In general, quadrature/cubature formulae of the form \eqref{eq:def_quadrature_general} might not be provably stable to perturbations in the evaluations of $\fexp$ at the nodes.   
Denoting with $\eta_\indmeas$ the perturbation of $\fexp(\y_\indmeas)$, for the formula \eqref{eq:def_quadrature_general} it holds 
\begin{equation}
\label{eq:suff_cond_stab2}
\left|
\sum_{\indmeas=1 }^\M \quadweights_\indmeas ( \fexp(\y_\indmeas) + \eta_\indmeas)  -
\sum_{\indmeas=1 }^\M \quadweights_\indmeas \fexp(\y_\indmeas) \right| \leq 
\max_{j=1,\ldots,\M } |\eta_j |
\sum_{\indmeas=1 }^\M \left| \quadweights_\indmeas \right|.  
\end{equation}
As long as $V_n$ contains the functions that are constant over $\Gamma$, exactness of \eqref{eq:def_quadrature_general} over $V_n$ implies
\begin{equation}
\sum_{\indmeas=1 }^\M \quadweights_\indmeas = 1. 
\label{eq:exact_weights_cond}
\end{equation}
On the one hand, in presence of negative weights the summation of the $ \left| \quadweights_\indmeas \right|$ in \eqref{eq:suff_cond_stab2} can become larger than one, thus serving as an amplifying factor for the perturbations.    
On the other hand, if the weights are positive,  \eqref{eq:suff_cond_stab2} and \eqref{eq:exact_weights_cond} give 
\begin{equation}
\label{eq:suff_cond_stab}
\left|
\sum_{\indmeas=1 }^\M \quadweights_\indmeas ( \fexp(\y_\indmeas) + \eta_\indmeas)  -
\sum_{\indmeas=1 }^\M \quadweights_\indmeas \fexp(\y_\indmeas) \right| \leq 
\max_{\indmeas=1,\ldots,\M } |\eta_\indmeas |,   
\end{equation}
thus ensuring the stability of the formula \eqref{eq:def_quadrature_general} to perturbations. 

In one dimension, the weights of Gaussian quadratures are strictly positive. In higher dimension, cubatures with positive weights are difficult to construct, above all in general domains or with general densities $\density$. A remarkable result on the existence of stable quadrature/cubature formulae is the next theorem from \cite{T1957}. 

\begin{theorem} 
\label{thm:tchakaloff}
Let $\Gamma \subset \mathbb{R}^\s$ be a compact set, and consider the integral \eqref{eq:integral_analytic} with $\density$ strictly positive over $\Gamma$. 
Given $n$ real functions $f_1,\ldots,f_n$ that are continuous on $\Gamma$, linearly independent, and such that at least one is nonzero everywhere in $\Gamma$, there exists $(\y_\indmeas)_{\indmeas=1}^{\M} \subset \Gamma$ and nonnegative reals $(\quadweights_\indmeas)_{\indmeas=1}^\M$ with $\M\leq n$ such that the formula \eqref{eq:def_quadrature_general} for \eqref{eq:integral_analytic} is exact on $\textrm{span}\{f_1,\ldots,f_n \}$. 
\end{theorem}

The result in Theorem~\ref{thm:tchakaloff} gives an upper bound for the number of nodes, and can be further generalized to unbounded domains, see \emph{e.g.}~\cite{P1997,BT2006}. 
With classical spaces of algebraic polynomials over the hypercube $\Gamma=[0,1]^\s$, there are known lower bounds on the number of nodes required by any cubature formula for exactness, \emph{e.g.}~in \cite{Stroud1960} and in \cite{Moller1979} for polynomial spaces supported on isotropic tensor product or total degree index sets. In such specific settings, it is possible to construct exact cubature formulae whose number of nodes matches those lower bounds, so-called minimal cubature formulae, see \emph{e.g.}~\cite{Cools} for an overview. With more general spaces of algebraic polynomials, minimal cubature formulae are not known on a systematic basis. 
In low dimension $\s\leq 3$, heuristic numerical methods based on optimization as those proposed in \cite{RB} allow the direct construction of (almost) minimal cubature formulae. In high dimension it is difficult to find minimal cubature formulae, and it is difficult also to find cubature formulae whose number of nodes matches the upper bound in Theorem~\ref{thm:tchakaloff}, above all with general polynomial spaces and with general domains. It is worth noticing that the proof of Theorem~\ref{thm:tchakaloff} from \cite{T1957} is not constructive, and finding the nodes and weights  of such remarkable formulae remains an open problem.

We now briefly recall other approaches to multivariate integration from the literature, that use a cubature formula of the form \eqref{eq:def_quadrature_general} with suitable choices for the nodes and weights. A first approach, that requires a product domain and a separable density $\density$, 
has been developed starting from \cite{Smolyak1963}, and nowadays goes under the name of Smolyak cubature or sparse grid quadrature/cubature, see \emph{e.g.}~\cite{BG2004a} and references therein. In general, a sparse grid cubature formula is exact on a chosen polynomial space, see for example those presented in \cite{NR99}.   
In low dimension and using polynomials with moderately high (total) degree, another type of (symmetric) cubature formulae has been presented in \cite{MS1967,GenzKeister}, and these cubatures become unstable as the degree of the polynomials increases. 

Other approaches to multivariate integration are Monte Carlo and quasi-Monte Carlo methods, see \cite{RC,DKS2013a} and references therein. These approaches do not impose exactness of the cubature formula on a given space. Both Monte Carlo and quasi-Monte Carlo methods share the same weights, that are equal to $\M^{-1}$ with $\M$ being the number of nodes. The difference between the two methods is in the choice of the nodes. In the Monte Carlo method the nodes are realizations of independent and identically distributed copies of a random variable distributed as $\mu$. In the quasi-Monte Carlo method the nodes are judiciously chosen according to specific deterministic rules, and a tensor product domain again is required. The notorious half-order convergence rate of the Monte Carlo method is immune to the curse of dimensionality. The convergence rate of quasi-Monte Carlo depends on the structure and smoothness of the integrand function, and on the low-discrepancy properties of the point set containing the nodes. 

In the present article we develop cubature formulae of the form \eqref{eq:def_quadrature_general} that are exact on $V_n$, for general domains $\Gamma\subseteq \Rset^\s$ provided an $L^2(\Gamma,\mu)$ orthonormal basis is available, and whose weights and nodes can be explicitly calculated. To pursue such an objective, we replace the integrand function by its weighted least-squares approximation onto $V_n$. Our cubature formulae use randomized nodes, and their exactness on $V_n$ occurs with a quantified large probability. More precisely, see Theorem~\ref{thm:estimates_quadrature}, if $\M$ is linearly proportional to $n$, up to a logarithmic factor, then the formula is exact on $V_n$ with large probability.
This shows that exact cubature formulae can be constructed in the general setting of Theorem~\ref{thm:tchakaloff} with $\M$ of the order of $n$, albeit not necessarily with positive weights. 

Given a function $\fexp$ in some smoothness class, for the proposed cubatures we provide convergence estimates in probability and in expectation for the integration error   
\begin{equation}
\vert   I(\fexp) - \Q_{\M}(\fexp) \vert
\label{eq:integration_error}
\end{equation}
depending on $\M$, $n$, and on the best approximation error of $\fexp$ in $V_n$ in some norm. 
In particular, we show that the \emph{mean error} satisfies an estimate of the type 
\begin{equation}
\E( |I(\fexp)-I_\M(\fexp) | ) \lesssim \sqrt{ \dfrac{n}{\M}  } \inf_{ \genpol \in V_n } \| \fexp - \genpol \|_{L^2(\Gamma,\mu)} + \M^{-\param},
\label{eq:type_res_one}
\end{equation}
provided $\frac{ \M }{ \ln \M } \gtrsim (1+\param) n$, where $\param >0$ can be taken arbitrarily and all hidden constants do not depend on $\s$. Similar estimates are proved for the mean squared error and in probability, and use recent results from \cite{CM2016} on the analysis of the stability and accuracy of weighted least-squares approximation methods. 

The previous estimate shows that if $n$ is kept fixed and $\M \to \infty$, the cubature formula has a Monte Carlo type convergence rate, however with a substantially reduced variance with respect to a standard Monte Carlo cubature formula, proportional to the $L^2(\Gamma,\mu)$ projection error of $\fexp$ onto $V_n$. In this respect, our cubature formula can be seen as a very efficient variance reduction technique in a Monte Carlo context. 

Moreover, if both $n,\M\to \infty$, still under the condition $\frac{\M}{\ln \M} \gtrsim n$, then the mean cubature error is of the same order of the $L^2(\Gamma,\mu)$-best approximation error and features a fast decay if the integrand function is smooth and the sequence of finite-dimensional spaces $V_n$ has good approximation properties. In this respect our cubature formula is of high order. We remark once more that all results and hidden constants do not depend on the dimension $\s$. 

As a further contribution, we construct another cubature formula for $I(\fexp)$  consisting of $I_\M(\fexp)$ plus a Monte Carlo correction estimator.   
This cubature satisfies an error estimate like \eqref{eq:type_res_one} without the term $\sqrt{n}$, and hence is asymptotically optimal in $n$, but with different 
constants that can be larger in the preasymptotic regime. 
Such a result is stated in Theorem~\ref{thm:estimates_quadrature_asy}, and is made possible by using 
 $2\M$ nonidentically distributed nodes 
 rather than $\M$ identically distributed nodes. 

An often desirable property of the quadrature/cubature formula \eqref{eq:def_quadrature_general} concerns the positivity of all the weights, that ensures stability as shown in \eqref{eq:suff_cond_stab}. This property is not fulfilled by sparse grid cubature formulae, which may contain negative weights, unless tensor grids are used and the underlying quadrature has positive weights (\emph{e.g.}~Gaussian quadratures).  
In Theorem~\ref{thm:positive_weights} we show that cubature formulae exact on $V_n$ and with strictly positive weights can be constructed, again in the same general setting of Theorem~\ref{thm:tchakaloff}, but with $\M$ being at least quadratically proportional to $n$, up to a logarithmic factor. At present time we are not aware of the existence in the literature of stable cubature formulae in general domains which have simultaneously high-degree of exactness and strictly positive weights, and the present paper provides a first analysis of cubature formulae of this type. 

Cubature formulae exact on a given space and  with nodes being random variables have been earlier studied in \cite{EZ1960,E1975}, but the weights are not necessarily positive. 
Error estimates for these stochastic cubature formulae can be found in 
\cite[Lemma 1]{N88}, which can be seen as a stochastic counterpart of Theorem~\ref{thm:tchakaloff}.  
The stochastic cubatures from \cite{EZ1960,E1975} use $n$ cubature nodes whose probability distribution contains determinants of $n$-by-$n$ matrices. 
For this reason, the generation of the nodes of such cubatures becomes computationally intractable already for moderate values of $n$ and $d$.
Compared to 
\cite{EZ1960,E1975}, 
the stochastic cubature 
developed in the present paper use a different distribution of the nodes, 
that can be sampled way more efficiently. 
Under very mild requirements, for example if an $L^2(\Gamma,\mu)$-orthonormal basis of $V_n$ is known in explicit form, efficient algorithms 
have been developed in \cite{CM2016} 
for the generation of the nodes of our stochastic cubatures.  
If the elements of the orthonormal basis have product form, then the computational complexity of such algorithms for the generation of $\M$ nodes is provably linear in $\M$ and $\s$.

More recently, a cubature formula based on least-squares approximants for periodic functions
has been proposed in \cite{Ott}, with different weights,
different distribution of the nodes
 and a different error analysis. 
The cubature from \cite{Ott} has been later refined in \cite{KUP}.

The outline of the article is the following: in Section~\ref{sec:dls_pol_spaces} we recall some useful results on the analysis of discrete least squares. In Section~\ref{sec:multi_h_o_quad} we present the cubature formulae, and provide conditions which ensure exactness and positive weights, together with convergence estimates. Section~\ref{sec:four} addresses the case where $V_n$ is chosen as a multivariate polynomial space. 
The proofs are collected in Section~\ref{sec:proofs}. In Section~\ref{sec:six} we draw some conclusions.

\section{Discrete least-squares approximation}
\label{sec:dls_pol_spaces}
\noindent
In this section we introduce the weighted discrete least-squares method with evaluations at random point sets, and recall the main results achieved in \cite{CM2016} for the analysis of its stability and accuracy.  

Given a Borel probability measure $\mu$ on $\Gamma$, we introduce the $L^2(\Gamma,\mu)$ inner product 
\begin{equation}
\langle \genfun_1, \genfun_2 \rangle := \int_\Gamma \genfun_1(\y) \genfun_2(\y) \density(\y) d\lambda(\y),
\label{eq:cont_inn_prod}
\end{equation}
and the 
norm $\Vert \genfun \Vert:= \langle \genfun , \genfun \rangle^{1/2}$. 
We work under the following assumption. 
\begin{assumption}
\label{thm:assumption}
There exists an $L^2(\Gamma,\mu)$ orthonormal basis $(\basisfuncexp_j)_{j\geq 1}$, and this basis contains the constant function over $\Gamma$. 
\end{assumption}

Using the orthonormal basis $(\basisfuncexp_j)_{j\geq 1}$ we define the approximation space as 
\begin{equation}
V_n:=\textrm{span}\{\basisfuncexp_1,\ldots,\basisfuncexp_n\}, 
\label{eq:def_V_n}
\end{equation}
and set $n:=\textrm{dim}(V_n)$. 
Without loss of generality we suppose that $\basisfuncexp_1\equiv 1$, and therefore $\basisfuncexp_1 \in V_n$ for any $n\geq 1$, as stated in the next assumption.  
\begin{assumption}
For any $n\geq 1$ the space $V_n$ contains $\basisfuncexp_1\equiv 1$.
\label{thm:assumptionbis}
\end{assumption}
For any given space $V_n$ we define the functions $\kappa:\Gamma \to \mathbb{R}$ and $w:\Gamma \to \mathbb{R}$ as 
\begin{equation}
\kappa(\y):=\left(\sum_{\indmeas =1}^n | \basisfuncexp_\indmeas(\y) |^2 \right)^{-1} \textnormal{ and } w(\y):=n \, \kappa(\y),
\label{eq:weight_function}
\end{equation}
whose denominators do not vanish thanks to Assumption~\ref{thm:assumptionbis}. 
For any space $V_n$ and any $n\geq 1$, Assumption~\ref{thm:assumptionbis} ensures the upper bound 
\begin{equation}
w(\y) \leq n, \quad \y \in \Gamma.
\label{eq:upper_bound_w}
\end{equation}
The function $\kappa$ is strictly positive over $\Gamma$. 
Sharper lower bounds (uniformly over $\Gamma$) can be obtained by exploiting the structure of the space $V_n$: for example with polynomial spaces such lower bounds are shown in Remark~\ref{sec:lower_bound_polynomials}. 

When $V_n$ is the space of algebraic polynomials of total degree $n-1$, the function $\kappa$ is known as the Christoffel function.
Using such functions, we define on $\Gamma$ the probability measure 
\begin{equation}
\label{eq:prob_meas_aux}
d\sigma := w^{-1} d\mu
= \dfrac{ \sum_{\indmeas =1}^n | \basisfuncexp_\indmeas(\y) |^2  }{n} d\mu
= \dfrac{ \sum_{\indmeas =1}^n | \basisfuncexp_\indmeas(\y) |^2  }{n} \rho(\y) d\lambda. 
\end{equation} 
In general $\sigma$ is not a product measure, even if $\mu$ is a product measure.
Next, we introduce the weighted discrete inner product 
\begin{equation}
\langle \genfun_1, \genfun_2 \rangle_\M := \dfrac{1}{\M}  \sum_{\indmeas = 1 }^{\M} w(y_\indmeas) \genfun_1(\y_\indmeas)\genfun_2(\y_\indmeas), 
\label{eq:disc_inn_prod}
\end{equation}
where the functions $w,\genfun_1,\genfun_2$ are evaluated at $\M$ points $\y_1,\ldots,\y_\M\in\Gamma$ that are independent and identically distributed according to $\sigma$. 
The discrete inner product is associated with the discrete seminorm $\Vert \genfun\Vert_\M:=\langle\genfun ,\genfun\rangle_\M^{1/2}$ 
for any $\genfun \in L^2(\Gamma,\mu)$.  
In the forthcoming sections, the random points $\y_1,\ldots,\y_\M \in \Gamma$ play the role of   nodes for the quadrature/cubature formula \eqref{eq:def_quadrature_general}.

For any function $\fexp:\Gamma\to \Rset$, we define its $L^2(\Gamma,\mu)$ projection onto $V_n$ as 
\begin{equation}
\Pi_n \fexp := \argmin_{\genpol \in V_n} \Vert \genpol - \fexp \Vert.
\label{eq:projection_L2}
\end{equation}

In many applications we do not have an explicit expression of the function $\fexp$, and can only evaluate its value $\fexp(\y)$ at a given parameter $\y \in \Gamma$. 
In such a situation, the projection \eqref{eq:projection_L2} cannot be computed, since it would require the explicit knowledge of the function $\fexp$. Hence, one can resort to the discrete least-squares approximation of $\fexp$ in $V_n$, defined as
\begin{equation}
\Pi_n^\M \fexp := \argmin_{\genpol \in V_n} \Vert \genpol - \fexp \Vert_\M,
\label{eq:projection_discrete_least_squares}
\end{equation}
where the minimization of the $L^2(\Gamma,\mu)$ norm has been replaced by the minimization of the discrete seminorm. 
Since the discrete seminorm uses pointwise evaluations of $\fexp$, throughout the paper we further assume that $\fexp(\y)$ is well defined at any $\y \in \Gamma$. 
The expansion of $\Pi_n^\M \fexp$ over the orthonormal basis reads  
\begin{equation}
\Pi_n^\M \fexp = \sum_{j=1}^n \coef_j \basisfuncexp_j,
\label{eq:expansion_pol_space}
\end{equation}
with $\coef:=(\coef_1,\ldots,\coef_n)^\top \in \mathbb{R}^n$ being the vector of the coefficients in the above expansion. 
Denote with $\design$ the design matrix, whose elements are defined as   
\begin{equation}
\label{eq:def_D}
\design_{\indmeas\indbasis}:= \sqrt{w(\y_\indmeas)} \basisfuncexp_\indbasis(\y_\indmeas), \quad \indmeas=1,\ldots,\M, \quad  \indbasis=1,\ldots,n, 
\end{equation}
and define the Gramian matrix $\gramian$ as
\[
\gramian := \dfrac{1}{\M} \design^\top \design.
\]

Moreover, we denote the Hadamard product of two vectors $p,q \in \Rset^\M$ as 
$$
p \odot q 
:= (p_1, \ldots, p_\M) \odot (q_1,\ldots,q_\M) 
= (p_1 q_1,\ldots,p_\M q_\M) \in \Rset^\M,   
$$
and use this product to define  
$\rhs:= 
\sqrt{w}
\odot \rhsfun=( \sqrt{w(\y_1)} \fexp(\y_1),\ldots,\sqrt{w(\y_\M)} \fexp(\y_\M) )^\top$,  
where $\rhsfun:=( \fexp(\y_1),\ldots,\fexp(\y_\M))^\top$ contains the evaluations of $\fexp$ at the nodes $(\y_\indmeas)_{1 \leq \indmeas \leq \M}$.

The projection $\Pi_n^\M \fexp$ of $\fexp$ in \eqref{eq:projection_discrete_least_squares} can be computed by solving the linear system 
\begin{equation}
\gramian
\coef = \dfrac{1}{\M} \design^\top \rhs. 
\label{eq:normal_equations}
\end{equation}
If the matrix $\gramian$ is nonsingular then the solution to the linear system \eqref{eq:normal_equations} is 
\begin{equation}
\coef =  \dfrac{1}{\M} \gramian^{-1} \design^\top \rhs,  
\label{eq:coefficients}
\end{equation}
thus defining a unique discrete least-squares approximation of $\fexp$ through \eqref{eq:expansion_pol_space}.
For any integer $n \geq 1$, we say that a point set $\y_1,\ldots,\y_\M \in \Gamma$ with $\M\geq n$ is unisolvent for a given space $V_n$ if 
\begin{equation}
\textrm{det}(\gramian)\neq 0.
\label{eq:cond_gramian_nonsing}
\end{equation} 
A unisolvent point set ensures that the operator $\Pi_n^\M$ is well defined and uniquely associated to the space $V_n$. 
When $\M < n$ the matrix $\gramian$ is rank deficient, and condition \eqref{eq:cond_gramian_nonsing} cannot be fulfilled. 
Condition \eqref{eq:cond_gramian_nonsing} does not depend on the choice of the basis of $V_n$: using any other basis $(\widetilde{\basisfuncexp}_1,\ldots,\widetilde{\basisfuncexp}_n )^\top = M \, (\basisfuncexp_1,\ldots,\basisfuncexp_n )^\top$ of $V_n$ related to the $\basisfuncexp_1,\ldots,\basisfuncexp_n$ by means of a suitable nonsingular matrix $M$ yields 
$\textrm{det}(M^\top \gramian M) \neq 0 \iff \textrm{det}(\gramian)  \neq 0 $.   
Given any matrix $A\in\Rset^{m\times n}$, for any $1 \leq p \leq +\infty$ we introduce the operator norm   
\[
\vvvert  A  \vvvert_{\ell_p} := \sup_{x \in \Rset^n \atop x \neq 0} \dfrac{ \Vert A x   \Vert_{\ell_p} }{ \Vert x \Vert_{\ell_p}},
\]
and define $\vvvert \cdot \vvvert:=\vvvert \cdot \vvvert_{\ell_2}$ when $p=2$. 
Condition \eqref{eq:cond_gramian_nonsing}
does not take into account situations where $\gramian$ is nonsingular but still very ill-conditioned, which might occur even when $\M \geq n$. 
From a numerical standpoint, it is desirable that $\gramian$ is also well conditioned. 
A natural vehicle for quantifying the ill-conditioning of $\gramian$ is to look at how much it deviates from the identity matrix $\cI$ with compatible size, 
\begin{equation}
\vvvert  \gramian - \cI \vvvert \leq \delta, 
\label{eq:deviation_identity}
\end{equation}
for some $\delta>0$. 
When $\delta\in(0,1)$ condition \eqref{eq:deviation_identity} can be rewritten as the norm equivalence 
\begin{equation}
(1-\delta) \Vert \genpol \Vert^2 
\leq
\| \genpol \|_{\M}^2
\leq
(1+\delta) \| \genpol \|^2, 
\quad
\genpol \in V_n, 
\label{eq:norm_equivalence_norms}
\end{equation}
or as 
\begin{equation}
1-\delta
\leq 
\vvvert\gramian \vvvert
\leq 
1+\delta.
\end{equation}

We now recall the main results achieved in \cite{CM2016} concerning the analysis of the stability and accuracy of weighted discrete least-squares approximation with evaluations at random points. 
For any $\fexp\in L^2(\Gamma,\mu)$ we define its best approximation error in the $L^2(\Gamma,\mu)$ norm as
$$
e_2(\fexp):=\| \fexp - \Pi_n \fexp \| \tcheck{= \min_{v\in V_n} \| \fexp - v \|,} 
$$
and its weighted $L^\infty$ best approximation error as 
$$
e_{\infty,w}(\fexp):=\inf_{\genpol \in V_n }   
\sup_{\y \in \Gamma}  \sqrt{w(\y)} | \fexp(\y) - \genpol(\y) |.  
$$
Given any $\delta \in (0,1)$ we define the quantity  
\begin{equation}
\label{eq:def_cost_delta}
\cdelta(\delta):= (1+\delta) \ln\left(1+\delta\right) -\delta > 0,  
\end{equation}
that satisfies the upper and lower bounds in \eqref{eq:upper_bound_zeta_delta} and \eqref{eq:lower_bound_zeta_delta}. 

Also recall the conditioned weighted least-squares estimator introduced in \cite{CM2016}, defined as 
$$
\widetilde{\fexp} := 
\begin{cases}
\Pi_n^\M \fexp, & \textrm{ if }  \vvvert \gramian - \cI \vvvert < \delta, \\
0, & \textrm{ otherwise. } 
\end{cases}
$$

Next we quote from \cite[Corollary 1]{CM2016} the following result. 
\begin{theorem}
\label{thm:dls_multivariate}
In any dimension $\s$, for any real $\param > 0$, any $\delta \in (0,1)$ and any $n\geq 1$, if the $\M$ i.i.d.~points $\y_1,\ldots,\y_\M$ are drawn from $\sigma$ defined in \eqref{eq:prob_meas_aux} and 
\begin{equation}
\dfrac{\M}{\ln \M} \geq \dfrac{ 1 + \param }{ \cdelta(\delta) } n, 
\label{eq:condition_points}
\end{equation}
then the following holds:
\begin{itemize}
\item[(i)] the matrix $\gramian$ satisfies 
\begin{equation}
\label{eq:statement_gramian}
\Pr\left(   \vvvert \gramian - \cI \vvvert \leq \delta   \right) 
>
1-  2 n  \M^{-(\param+1)}
\geq  
1-  2 \M^{-\param}; 
\end{equation}
\item[(ii)] for all $\fexp$ such that $\sup_{\y \in \Gamma}  \sqrt{w(\y)} | \fexp(\y) |<+\infty$ the weighted least-squares estimator satisfies 
\begin{equation}
\label{eq:prob_estimate}
\Pr\left(  \|   \fexp  - \Pi_n^m \fexp  \|  \leq \left( 1+\sqrt2  \right) 
e_{\infty,w}(\fexp)
 \right)  
> 
1-  2 n  \M^{-(\param+1)}
\geq 
1- 2\M^{-\param}; 
\end{equation}
\item[(iii)] if $\fexp \in L^2(\Gamma,\mu)$ then the conditioned estimator satisfies  
\begin{equation}
\label{eq:expect_estimate}
\mathbb{E}\left(  \|   \fexp  - \widetilde{\fexp}  \|^2  \right) \leq \left( 1 + \varepsilon(\M)  \right)  (e_2(\fexp))^2 + 2 \| \fexp \|^2 \M^{-\param}, 
\end{equation}
where 
\begin{equation*}
\varepsilon(\M):=\dfrac{ 4\cdelta(\delta) }{(1+\param) \ln \M}
\end{equation*}
decreases monotonically to zero as $\M$ increases. 
\end{itemize}
\end{theorem}

The bound \eqref{eq:deviation_identity} for $\gramian$ implies a bound of the same type for its inverse: for any $\delta \in (0,1)$ and $\M\geq n$ it holds that 
\begin{equation}
 \vvvert \gramian - \cI \vvvert \leq \delta \implies \vvvert \gramian^{-1} - \cI  \vvvert\leq \dfrac{\delta}{1-\delta}, 
\label{eq:inverse_gramian}
\end{equation}
since 
$$
\vvvert  \gramian^{-1} - \cI \vvvert  = \vvvert \gramian^{-1} \left( \cI - \gramian^{}  \right) \vvvert  \leq  \vvvert \gramian^{-1} \vvvert \ \vvvert  \gramian^{} - \cI \vvvert, 
$$
and \eqref{eq:deviation_identity} also implies $\vvvert  \gramian \vvvert\leq 1+\delta$ and $\vvvert  \gramian^{-1} \vvvert \leq (1-\delta)^{-1}$. 
Another implication of \eqref{eq:deviation_identity} is that 
\begin{equation}
\vvvert \gramian - \cI \vvvert \leq \delta \implies 
\textrm{cond}(\gramian) \leq \dfrac{1+\delta}{1-\delta}. 
\label{eq:bound_cond_number}
\end{equation}

From \eqref{eq:statement_gramian} together with \eqref{eq:inverse_gramian} we have the following corollary of Theorem~\ref{thm:dls_multivariate}. 
\begin{corollary}
\label{thm:dls_multivariate_coro}
In any dimension $\s$, for any real $\param > 0$, any $\delta \in (0,1)$ and any $n \geq 1$, if the $\M$ i.i.d.~points $\y_1,\ldots,\y_\M$ are drawn from $\sigma$ defined in \eqref{eq:prob_meas_aux} and condition  \eqref{eq:condition_points} holds true then 
\begin{equation}
\label{eq:statement_inverse}
\Pr\left(   \vvvert \gramian^{-1} - \cI \vvvert \leq \dfrac{\delta }{1-\delta}  \right) > 1-  2 \M^{-\param}.
\end{equation}
\end{corollary}

Since \eqref{eq:deviation_identity} implies that $\gramian$ is nonsingular, another consequence of Theorem~\ref{thm:dls_multivariate} is the next corollary. 

\begin{corollary}
\label{thm:coro_well_conditioned}
In any dimension $\s$, for any real $\param>0$, any $\delta \in (0,1)$ and any $n\geq 1$, if $\M$ satisfies \eqref{eq:condition_points} then the random point set  $\y_1,\ldots,\y_\M$ with $\M$ i.i.d.~points drawn from $\sigma$ is unisolvent for $V_n$ with probability larger than $1-2\M^{-\param}$. 
\end{corollary}

Sampling algorithms for the generation of the random point set $\y_1,\ldots,\y_\M$ from \eqref{eq:prob_meas_aux} have been developed in \cite{CM2016} when $\Gamma$ is a Cartesian domain and $\mu$ is a product measure. The computational cost of these algorithms scales linearly with respect to the dimension $\s$ and to $\M$.

\section{Randomized high-order cubature formulae}
\label{sec:multi_h_o_quad}
\noindent
In this section we construct cubature formulae of the form \eqref{eq:def_quadrature_general}, to approximate the multivariate integral \eqref{eq:integral_analytic} of a smooth function $\fexp:\Gamma\to\Rset$ with respect to a given probability measure $\mu$. 
The construction of such cubature formulae uses the discrete least-squares approximation \eqref{eq:projection_discrete_least_squares} of the integrand function $\fexp$. 
As in the previous sections, $\y_1,\ldots,\y_\M$ and $\quadweights_1,\ldots,\quadweights_\M$ denote its nodes and weights, respectively.  
The first step in the development of our cubature formula consists in the evaluation $\fexp(\y_1),\ldots,\fexp(\y_\M)$ of the integrand function $\fexp$ at the nodes $\y_1,\ldots,\y_\M$.  
The second step is the choice of the weights. 
Let $W$ be the matrix defined element-wise as $W_{ii}:=w(\y_i)$ for $i=1,\ldots,\M$ and $W_{ij}=0$ for $i,j=1,\ldots,\M$ with $i\neq j$. 
We consider weights $\quadweights=(\quadweights_1,\ldots,\quadweights_\M)^\top$ of the form
\begin{equation}
\quadweights := 
(\sqrt{w(\y_1)},\ldots,\sqrt{w(\y_\M)})^\top
\odot 
\dfrac{1}{\M} \design \gramian^{-1} e_1= 
\dfrac{1}{\M}
W^{1/2}
 \design \gramian^{-1} e_1, 
\label{eq:def_weight_least_squares}
\end{equation}
where $e_1:=(1,0,\ldots,0)^\top \in \Rset^{n}$ denotes the vector with all components except the first being equal to zero, and the first component being equal to one. 
The components of $\quadweights$ are given by 
\begin{equation}
\quadweights_\indmeas 
:= 
\dfrac{1}{\M} \left( W^{1/2} \design \gramian^{-1} e_1 \right)_\indmeas
= 
\dfrac{1}{\M} 
\sqrt{
w(\y_\indmeas)
} 
\design_\indmeas \gramian^{-1} e_1, 
\quad  \indmeas=1,\ldots,\M,
\label{eq:weight_least_squares}
\end{equation}
with 
\[
\design_\indmeas := \sqrt{w(\y_\indmeas)} \left( \basisfuncexp_1(\y_\indmeas),\ldots, \basisfuncexp_{n}(\y_\indmeas) \right). 
\]

In the next lemma we prove conditions on the nodes and weights such that the cubature  \eqref{eq:def_quadrature_general} is exact on $V_n$,  
see Section~\ref{sec:proofs} for the proof.  

\begin{lemma}
\label{thm:lemma_weights_exactness}
In any dimension $\s \geq 1$ and for any $n \geq 1$, let $\M\geq n$  nodes $\y_1,\ldots,\y_\M\in\Gamma$ be a unisolvent point set for the space $V_n$. 
If the weights $\quadweights_1,\ldots,\quadweights_\M$ are chosen as in \eqref{eq:def_weight_least_squares} then the formula  \eqref{eq:def_quadrature_general} satisfies 
\begin{equation}
\Q_\M(\fexp)=I(\Pi_n^\M \fexp), \quad \textrm{for any } \fexp \in L^2(\Gamma,\mu), 
\label{eq:lemma_ex_phi}
\end{equation}
and
\begin{equation}
\Q_\M(\genpol)=I(\genpol) , \quad \textrm{for any }\genpol \in V_n. 
\label{eq:lemma_ex_pol}
\end{equation}
\end{lemma}

\begin{remark}
\label{thm:computation_weights}
The weights in equation \eqref{eq:def_weight_least_squares} can be calculated as $\quadweights=\frac{1}{\M} W^{1/2} \design h$, where $h$ is the solution to the linear system $\gramian h =e_1$. 
Note that, from Corollary~\ref{thm:dls_multivariate_coro}, 
the matrix $\gramian$ is well conditioned under condition \eqref{eq:condition_points}. 
\end{remark}

For any $n\geq 1$, 
the projection $\Pi_n \fexp$ of $\fexp$ onto $V_n$ satisfies 
\begin{equation}
\label{eq:integral_proj_ex}
 I( \Pi_n \fexp)=I(\fexp), \quad \fexp \in L^2(\Gamma,\mu),  
\end{equation}
whose proof is identical to the proof of \eqref{eq:lemma_ex_phi}. 
Using \eqref{eq:lemma_ex_pol} together with \eqref{eq:integral_proj_ex} it follows that  
\begin{equation}
\label{eq:quadrature_proj_ex}
\Q_\M( \Pi_n \fexp )=I(\fexp), \quad \fexp \in L^2(\Gamma,\mu).
\end{equation}
Set $h_n:=\fexp - \Pi_n \fexp$, and define the vector $\err\in \mathbb{R}^\M$ whose components are given by $\err_\indmeas=h_n(\y_\indmeas)$ for any $\indmeas=1,\ldots,\M$. 
We can decompose the integration error of $\Q_\M$ into two terms named $S$ and $B$ respectively: 
\begin{align}
\nonumber 
\Q_\M(\fexp)- I(\fexp)
 & = \int_\Gamma \left(  \Pi_n^\M \fexp -\Pi_n \fexp  \right) \, d\mu  \\
\nonumber 
& = \int_\Gamma \Pi_n^\M \left( \fexp - \Pi_n \fexp \right) \, d\mu \\
\label{eq:bias_var_Q_m_mat_nodec}
& = \dfrac{1}{\M} e_1^\top \gramian^{-1} \design^\top W^{1/2} \err \\
\label{eq:bias_var_Q_m_mat}
& = \dfrac{1}{\M} e_1^\top  \design^\top W^{1/2} \err + \dfrac{1}{\M} e_1^\top ( \gramian^{-1} - \cI) \design^\top W^{1/2} \err \\
& = \underbrace{\dfrac{1}{\M} \sum_{\indmeas=1}^{\M} w(\y_\indmeas)  h_n(\y_\indmeas)}_{=:S} 
+  
\underbrace{ \dfrac{1}{\M} \sum_{\indmeas=1}^\M \left( 
e_1^\top \gramian^{-1} (\cI - \gramian) \design^\top W^{1/2}
\right)_\indmeas
 h_n(\y_\indmeas)  }_{=:B}, 
\label{eq:bias_var_Q_m}
\end{align}
where the first equality uses \eqref{eq:integral_proj_ex} and \eqref{eq:lemma_ex_phi}, the second equality follows from properties of projection operators, the third equality uses \eqref{eq:lemma_ex_phi} together with the definitions of $\Q_\M$ and $\err$. 
Notice that \eqref{eq:integral_proj_ex} implies 
\begin{equation}
\E(S)=\E\left( w \fexp \right) - \E\left( w \Pi_n \fexp \right) =0, 
\label{eq:zero_mean_err_w}
\end{equation}
but there is no reason for $B$ to have zero mean, and in general 
$$
\E(B)=\E(\Q_\M(\fexp)) - I(\fexp) \neq 0. 
$$
Usually $\mathbb{B}:=\E(B)$ is called the \emph{bias} of $\Q_\M$.    

By conditioning depending on the value of $\vvvert \gramian-\cI \vvvert $, for any $\delta \in (0,1)$ we can define another cubature formula
\begin{equation}
\widetilde \Q_\M(\fexp):= 
\begin{cases}
\Q_\M(\fexp)
, & \textrm{ if }  \vvvert \gramian - \cI \vvvert < \delta, \\
0, & \textrm{ otherwise, } 
\end{cases}
\label{eq:quadrature_conditioned}
\end{equation}
that uses the cubature formula $\Q_\M(\fexp)$ defined in 
\eqref{eq:def_quadrature_general} 
with weights \eqref{eq:def_weight_least_squares}. 
Notice that $\widetilde\Q_\M(\fexp)$ can also be written as a cubature formula of the form \eqref{eq:def_quadrature_general},  
with the same nodes as $\Q_\M(\fexp)$
but with weights given by 
\begin{equation}
\quadweights = 
\begin{cases}
\dfrac{1}{\M} 
W^{1/2}
\design \gramian^{-1} e_1,
& \textrm{ if }  \vvvert \gramian - \cI \vvvert < \delta, \\
(0,\ldots,0)^\top, & \textrm{ otherwise. } 
\end{cases}
\label{eq:quadrature_conditioned_weights}
\end{equation}

A consequence of \eqref{eq:lemma_ex_phi} is that $\widetilde \Q_\M(\fexp)=I(\widetilde{\fexp})$ for any $\fexp \in L^2(\Gamma,\mu)$ on both events $\{\vvvert \gramian - \cI \vvvert < \delta\}$ and $\{\vvvert \gramian - \cI \vvvert \geq \delta\}$. 
However, \eqref{eq:lemma_ex_pol} with $\Q_\M(\genpol)$ replaced by $\widetilde \Q_\M(\genpol)$ remains true only on 
the event $\{\vvvert \gramian - \cI \vvvert < \delta\}$, because $\widetilde \Q_\M(\genpol)=0$ for any $\genpol \in V_n$ when $\{\vvvert \gramian - \cI \vvvert \geq \delta\}$.
On the event $\{\vvvert \gramian - \cI \vvvert < \delta\}$, 
the integration error of $\widetilde \Q_\M$ can be decomposed as 
\begin{equation}
I(\fexp)- \widetilde \Q_\M(\fexp) = S + B, 
\label{eq:bias_var_tilde_Q_m}
\end{equation}
where the terms $S$ and $B$ are the same that appear in \eqref{eq:bias_var_Q_m}. 
The bias of $\widetilde \Q_\M$ is $
\widetilde{
\mathbb{B}
}
:=\E(\widetilde \Q_\M(\fexp))-I(\fexp)$. 
The expectation in the definitions of $\mathbb{B}$ and $\widetilde{\mathbb{B}}$ is over both events $\{\vvvert \gramian - \cI \vvvert < \delta\}$ and $\{ \vvvert \gramian - \cI \vvvert \geq \delta\}$. 
In contrast to $\mathbb{B}$, for $\widetilde{\mathbb{B}}$ such an expectation is always finite. 
The term $\widetilde{\mathbb{B}}$ asymptotically vanishes as $\M\to +\infty$, 
and the proof of this fact is postponed to the end of this section. 
However $\mathbb{B}$ and $\widetilde{\mathbb{B}}$ do not vanish, in general, when $\M$ is finite. 

The next theorem quantifies the integration error of the formula \eqref{eq:def_quadrature_general} in probability, and the integration error of the formula \eqref{eq:quadrature_conditioned} in expectation. 
Its proof is postponed to Section~\ref{sec:proofs}. 

\begin{theorem}
\label{thm:estimates_quadrature}
In any dimension $\s$, for any real $\param > 0$ and any $\delta \in(0,1)$, if the $\M$ i.i.d.~points $\y_1,\ldots,\y_\M$ are drawn from $\sigma$ defined in \eqref{eq:prob_meas_aux} and condition \eqref{eq:condition_points} holds true then: 
\begin{itemize}
\item[i)] the cubature formula \eqref{eq:def_quadrature_general} with weights chosen as in \eqref{eq:def_weight_least_squares}
satisfies \eqref{eq:lemma_ex_phi}--\eqref{eq:lemma_ex_pol} with probability larger than $1-2\M^{-\param}$;
\item[ii)] for all 
$\fexp$ such that $ \sup_\y \sqrt{ w(\y)} | \fexp(\y) | < +\infty$, 
the integration error of the cubature formula \eqref{eq:def_quadrature_general} with weights \eqref{eq:def_weight_least_squares} satisfies
\begin{equation}
\label{eq:probabilistic_estimate}
\Pr\left(
 \left|  
I(\fexp)
-
\Q_{\M}(\fexp) 
  \right| 
\leq 
(1+ \sqrt 2)
e_{\infty,w}(\fexp)
\right) 
\geq 1- 2\M^{-\param};
\end{equation}
\item[iii)] for any $\fexp \in L^2(\Gamma,\mu)$ the integration error of the cubature formula  \eqref{eq:quadrature_conditioned} satisfies 
\begin{equation}
\label{eq:expectation_estimate}
\mathbb{E}
\left( 
\left|  
I(\fexp)
-
\widetilde \Q_{\M}(\fexp)
  \right|^2 
\right) 
\leq 
\left( 1 + \varepsilon(\M)  \right)  
(e_2(\fexp))^2
+ 2 | I(\fexp) |^2 \M^{-\param},
\end{equation}
with $\varepsilon(\M)$ as in Theorem~\ref{thm:dls_multivariate}, and also satisfies
\begin{equation}
\label{eq:expectation_estimate_L1_alt}
\mathbb{E}
\left( 
\left|  
I(\fexp)
-
\widetilde \Q_{\M}(\fexp) 
  \right| 
\right) 
\leq 
\sqrt{ 
\dfrac{n}{\M}
}
\left( 1 + 
\dfrac{ \varepsilon(\M,n) }{ 1-\delta  } 
\right)  
e_2(\fexp)
+ 2 
|I(\fexp)|  
\M^{-\param},
\end{equation}
with 
$$
\varepsilon(\M,n):= 
\sqrt{ 4(1+2 \lceil \ln n \rceil ) }
\left( 
1+ \sqrt{ 4(1+2 \lceil \ln n \rceil ) } \sqrt{ \dfrac{n}{\M} }
\right)
\leq 
\left( 
4+ 8 \lceil \ln n \rceil   
\right)
\left( 1+
\sqrt{ \dfrac{n}{\M} } 
\right).
$$
\end{itemize}
\end{theorem}

Before closing the section, we compare our randomized cubature formulas with the Monte Carlo method and with  Importance Sampling Monte Carlo, hereafter shortened to Importance Sampling.  
With all the three methods the integral $I(\fexp)$ in \eqref{eq:integral_analytic} is approximated using a cubature formula $\Q_{\M}(\fexp)$ as in \eqref{eq:def_quadrature_general}, but with different choices for the nodes and weights which amount to different estimates for the associated integration error \eqref{eq:integration_error}. 

With Monte Carlo the $\M$ nodes $\y_1,\ldots,\y_\M$ are independent and identically distributed  according to $\mu$, and the weights $\quadweights_1,\ldots,\quadweights_\M$ are all set equal to $1/\M$. The mean squared integration error of Monte Carlo is given by 
\begin{equation}
\label{eq:error_MC}
\E\left(  \left|  I(\fexp)
-
\Q_{\M}(\fexp)   \right|^2  \right) 
= 
\dfrac{ \textrm{Var}(\fexp) }{\M}. 
\end{equation}

With Importance Sampling, the $\M$ nodes $\y_1,\ldots,\y_\M$ are independent and can be chosen identically distributed according to $\sigma$, and the weights can be chosen as $\quadweights_\indmeas = w(\y_\indmeas)/\M$ for $\indmeas=1,\ldots,\M$. 
The corresponding mean squared integration error is given by 
\begin{equation}
\label{eq:error_IS_MC}
\E\left(  \left|  I(\fexp)
-
\Q_{\M}(\fexp)   \right|^2  \right) 
= 
\dfrac{ \textrm{Var}( w \fexp) }{\M}. 
\end{equation}

In our cubature formulas the $\M$ nodes $\y_1,\ldots,\y_\M$ are independent and identically distributed from $\sigma$, and the weights are either chosen as in \eqref{eq:def_weight_least_squares} or as in \eqref{eq:quadrature_conditioned_weights}. 
Concerning the error estimates, \eqref{eq:expectation_estimate} proves convergence in expectation of $|I(\fexp)- \widetilde \Q_{\M}(\fexp)|^2$ and \eqref{eq:probabilistic_estimate} proves a probabilistic estimate for the error $|I(\fexp)- \Q_{\M}(\fexp)|$.

Our cubature formulas and Importance Sampling both rely on the change of measure \eqref{eq:prob_meas_aux}, that is determined by the choice of the function $w$. 
In the present paper we have chosen $w$ as in \eqref{eq:weight_function}, but any other nonnegative function $w:\Gamma\to \mathbb{R}$ such that $\int_\Gamma w^{-1}\,d\mu=1$ could be used.
In Importance Sampling, the choice of $w$ should hopefully make the variance in \eqref{eq:error_IS_MC} smaller than the variance in \eqref{eq:error_MC} of Monte Carlo. In our cubature formula, taking $w$ as in \eqref{eq:weight_function} ensures the stability of the projector $\Pi_n^m$ as granted by Theorem~\ref{thm:dls_multivariate}.  

For any fixed $n$ and any $i=1,\ldots,\M$, the weights \eqref{eq:weight_least_squares} of the randomized cubature formula 
converge almost surely to the weights of importance sampling as $\M\to +\infty$.
Before presenting the proof of this result, we introduce the following notation:  
for any $\M\geq 1$ 
consider the finite sequence $(\y_1,\ldots,\y_\M) \subset \Gamma$ and define  
the matrix $\design^\M \in \Rset^{\M\times n}$, whose elements are 
$\design_{\indmeas k}^\M := \sqrt{w(\y_\indmeas)} \basisfuncexp_k(\y_\indmeas)$, and 
$$
\quadweights_{\indmeas,\M} := \dfrac{1}{\M} \sqrt{w(\y_\indmeas)} \design_\indmeas^\M (\gramian^{\M})^{-1} e_1, 
\quad
\gramian^\M := \dfrac{1}{\M} (\design^\M)^\top \design^\M,
\quad \indmeas =1,\ldots,\M,
$$
where $\design_\indmeas^\M$ is the $i$th row of $\design^\M$.  
The weights defined above are the same as in \eqref{eq:weight_least_squares}, and the purpose of the new notation is solely to emphasize the dependence on $\M$ in the design matrix $\design$ and in the cubature weights. 
\begin{theorem}
\label{thm:borel_cantelli}
Let $\overline{\M}$ satisfy \eqref{eq:condition_points} with some $\delta \in (0,1)$, $n\geq 1$, $\param>0$.
For any $\indmeas \in \mathbb{N}$
the sequence of random variables $(\M \quadweights_{\indmeas,\M})_{\M \geq \max\{ \indmeas, \overline{\M} \}}$  
converges almost surely to $w(\y_\indmeas)$ as $\M\to +\infty$. 
\end{theorem}
\begin{proof}
For any $\delta\in(0,1)$, define the sequence $\delta_k:=\delta 2^{-k}$, $k\geq 0$, and denote with $\M_k$
the number of samples that satisfies \eqref{eq:condition_points} with $\delta_k$, $n$ and $\param$.
For any $\indmeas \in \mathbb{N}$,
define  the events 
$$
A_k^i:=\left\{  \left| \M_k \quadweights_{\indmeas,\M_k} - w(\y_\indmeas)  \right| > \dfrac{ \delta_k \sqrt{w(\y_\indmeas) n }  }{1-\delta_k}  \right\}, \quad k\geq 0.  
$$
The left-hand side of the above inequality can be bounded as 
\begin{align*}
\left\vert \M_k \quadweights_{\indmeas, \M_k} - w(\y_\indmeas) \right\vert  
= 
\left\vert  \sqrt{ w(\y_\indmeas)}  \design_\indmeas^{\M_k} \left( (\gramian^{\M_k})^{-1} - \cI \right) e_1 \right\vert  
\leq 
\sqrt{w(\y_\indmeas)} 
\Vert \design_\indmeas^{\M_k} \Vert_{\ell_2}
\ 
\vvvert (\gramian^{\M_k})^{-1} - \cI \vvvert
\ 
\Vert e_1 \Vert_{\ell_2}
.
\end{align*}
For any $\indmeas=1,\ldots,\M_k$ we have $\Vert \design_\indmeas^{\M_k} \Vert_{\ell_2}^2 = w(y_i) \sum_{j=1}^{n} \basisfuncexp_j(\y_\indmeas)^2 =n$. 
Using \eqref{eq:statement_inverse}, under condition \eqref{eq:condition_points} with $\M_k$, $\delta_k$, $n$ and $\param$,  we have that
\[
\vvvert (\gramian^{\M_k})^{-1} - \cI \vvvert \leq \dfrac{\delta_k}{1-\delta_k},
\]
with probability at least $1-2n\M_k^{-(\param+1)}$. 
As a consequence $\Pr(A_k^i) \leq 2 n \M_k^{-(\param+1)}$ for any $k\geq 0$. 
Hence, using \eqref{eq:upper_bound_zeta_delta}, we have that 
\begin{align*}
\sum_{k\geq0} \Pr(A_k^i) \leq & 
 2 n^{-\param} 
(1+\param)^{-(\param+1)}
 \sum_{k\geq 0} \left( \dfrac{ \cdelta(\delta_k)
}{\ln \M_k} \right)^{\param+1} 
\\
\leq  &
 2 n^{-\param} \left( \dfrac{ 
\delta^2
}{ 2(1+\param)   } \right)^{\param+1} 
\sum_{k\geq 0} 
2^{-2k(\param+1) }
\\
= &
 2 n^{-\param} \left( \dfrac{ 
\delta^2
}{ 2(1+\param)   } \right)^{\param+1} 
\dfrac{1}{2^{2(\param+1)}-1}
< +\infty,  
\end{align*}
and the claim follows from Borel-Cantelli lemma. 
\end{proof}

In the limit $\M\to+\infty$, the matrix $\gramian$ tends almost surely to the $n$-by-$n$ identity matrix, see \emph{e.g.}~\cite[Theorem 1]{MNST2014} for a proof. 
From the strong law of large number we have 
$$
\dfrac{1}{\M} \rhsfun^\top W^{1/2} \design \stackrel{\M \to +\infty}{\to} \left( \int_\Gamma w \fexp \basisfuncexp_1 \, d\sigma ,\ldots, \int_\Gamma w \fexp \basisfuncexp_n \, d\sigma  \right)^\top, \quad \textrm{almost surely},  
$$
and as a consequence of Slutsky's theorem we also have the following almost sure convergence
\begin{equation}
\Q_{\M}(\fexp)
= 
\sum_{\indmeas=1}^\M  
\quadweights_\indmeas
\
\fexp(\y_\indmeas)
=
\rhsfun^\top \quadweights = \dfrac{1}{\M} \rhsfun^\top W^{1/2} \design  \gramian^{-1} e_1
 \stackrel{\M \to +\infty}{\to}  
\int_{\Gamma}
 w(\y) \fexp(\y) \, d\sigma 
=
\E(w \fexp)=
I(\fexp), \quad \fexp \in L^2(\Gamma,\mu).  
\label{eq:conv_asymp}  
\end{equation}
Using \eqref{eq:conv_asymp}, 
for any $\fexp\in L^2(\Gamma,\mu)$ we have 
\begin{equation*}
\E( \widetilde\Q_{\M}(\fexp) ) - I(\fexp) =  
\widetilde{\mathbb{B}}
 \stackrel{\M \to +\infty}{\to}  
0.
\end{equation*}
One can actually quantify more precisely the decay of $\widetilde{\mathbb{B}}$  w.r.t.~$\M$. 
Inspection of the proof of \eqref{eq:expectation_estimate_L1_alt}, more precisely using \eqref{eq:error_splitting_proof_est_L1}, \eqref{eq:unstable_term} and the notation from there, shows that there exist positive constants $C_1,C_2$ such that 
\begin{align*}
|
\widetilde{\mathbb{B}} 
|
\leq  
\int_{ \{\vvvert \gramian - \cI \vvvert \leq \delta\}  } | B | \, d\mu_\M
+
\int_{ \{\vvvert \gramian - \cI \vvvert > \delta\}  } | B | \, d\mu_\M
\leq 
\dfrac{n}{\M}
 \dfrac{ 
C_1+C_2 \ln n
}{1-\delta} e_2(\fexp)
+
2|I(\fexp)| \M^{-\param}.  
\end{align*}
Hence $\widetilde{\mathbb{B}}=\mathcal{O}(1/\M)$ showing that the bias term $\widetilde{\mathbb{B}}$ decays faster w.r.t.~$\M$ than the mean integration error in \eqref{eq:expectation_estimate_L1_alt}. 

The error estimate \eqref{eq:expectation_estimate} shows that, if $\param$ is large enough, then the root mean squared error decays at least as fast as the squared best approximation error. 
However the rate of convergence of this estimate does not catch up with those of Monte Carlo \eqref{eq:error_MC} and Importance Sampling \eqref{eq:error_IS_MC} due to the missing decay with respect to $\M$. 
Here the error estimate \eqref{eq:expectation_estimate_L1_alt} comes in handy since it recovers the same convergence rate of Monte Carlo and Importance Sampling with respect to $\M$. 
The estimate \eqref{eq:expectation_estimate_L1_alt} 
shows the main advantage of randomized cubatures
over 
Monte Carlo and Importance Sampling: 
the decay of the error \eqref{eq:expectation_estimate_L1_alt} 
is determined by the decay of the term $\M^{-1/2}$ and the decay of the best approximation error, 
in contrast to 
\eqref{eq:error_MC} and \eqref{eq:error_IS_MC} that only decay with respect to $\M$ at the same rate, \emph{i.e.}~$\M^{-1/2}$ for the root mean squared error.

In Theorem~\ref{thm:estimates_quadrature} we have written  
the estimate \eqref{eq:expectation_estimate_L1_alt} 
using the upper bound in Lemma~\ref{thm:bound_exp_abs_value}, 
that relates to the best approximation error in $L^2(\Gamma,\mu)$. 
The same estimate can be slightly improved using the equality \eqref{eq:error_similar_to_imp_samp}, that relates to the weighted best approximation error of $\fexp$ in $L^2(\Gamma,\mu)$. 
{\xc 
Since $w^{-1}\,d\mu$ in \eqref{eq:prob_meas_aux} 
is a probability measure and 
$w$ satisfies \eqref{eq:upper_bound_w}, 
}
such an error can be sandwiched 
for any $n\geq 1$ 
as 
$$
\| \fexp - \Pi_n \fexp \|_{L^1(\Gamma,\mu)}  \leq  \| \sqrt{w} (\fexp - \Pi_n \fexp) \| \leq \sqrt{n} \| \fexp - \Pi_n \fexp \|, \quad \fexp \in L^2(\Gamma,\mu), 
$$
or as

$$ 
n^{-1} \textrm{Var}( w(\fexp - \Pi_n \fexp) )
\leq \| \sqrt{w} (\fexp - \Pi_n \fexp ) \|^{\xc 2}
\leq \|w^{-1}\|_{L^\infty(\Gamma,\mu)} \textrm{Var}( w(\fexp - \Pi_n \fexp) )
, \quad \fexp \in L^2(\Gamma,\mu),  
$$
thanks to 
$\textrm{Var}( w(\fexp - \Pi_n \fexp) )=\| w (\fexp - \Pi_n \fexp ) \|^2$ from  \eqref{eq:zero_mean_err_w}.

For some polynomial approximation spaces, see the forthcoming Remark~\ref{sec:lower_bound_polynomials}, lower bounds for $w$ of the form \eqref{eq:lower_bound_w} are available. 
Using such a lower bound and taking $n=1$, the variance in the above formulas connects with the variance of importance sampling \eqref{eq:error_IS_MC}. 

In contrast to our cubature formulas, Importance Sampling and Monte Carlo are not exact cubature formulas on $V_n$. 

The more advantageous error estimate of randomized cubatures comes at the price of two additional computational tasks: the calculation of the cubature weights, that requires the solution of a linear system whose matrix $\gramian$ has size $n\times n$, see Remark~\ref{thm:computation_weights}, and the generation of the random samples from $\sigma_n$. 
The cost for solving the linear system does not depend on $\M$ and $\s$, and the cost for assembling $\gramian$ scales linearly in $\M$. The cost for the generation of the samples is provably linear with respect to $\s$ and $\M$, when $\Gamma$ is a product domain.

\begin{remark}[Adaptive randomized cubatures for nested sequences of approximation spaces]
\label{thm:remark_adaptive}
The results in Theorem~\ref{thm:estimates_quadrature} 
are proven using identically distributed random samples, and apply to 
any given approximation space $V_n$.  
The same result as Theorem~\ref{thm:estimates_quadrature} can be proven for another type of (nonidentically distributed) random samples, following the lines of the proof of \cite[Theorem~2]{M2018}.
These random samples 
allow the sequential construction of weighted least-squares estimators on any nested sequence of approximation spaces $(V_{n_k})_{k\geq 1}$ with dimension $n_k:=\textrm{dim}(V_{n_k})$, using an overall number of samples that remains linearly proportional to $n_k$, up to logarithmic terms.     
Using such a type of random samples and \cite[Theorem~3]{M2018}, the whole analysis of randomized cubatures from this article carries over using nested sequences of approximation spaces rather than a single space. 
\end{remark}

\begin{remark}
The error estimates \eqref{eq:probabilistic_estimate}, \eqref{eq:expectation_estimate}, 
\eqref{eq:expectation_estimate_L1_alt}, 
have been presented for real-valued functions, but they extend to complex-valued functions with essentially the same proofs by identifying $\mathbb{R}^2$ and $\mathbb{C}$. 
\end{remark}

\subsection{Randomized cubatures with optimal asymptotic convergence rate}
\label{sec:multi_h_o_quad_asy}
\noindent
In this section we analyse another cubature formula, 
that is obtained by adding a correction term to 
the the cubature $\widetilde \Q_\M(\fexp)$ 
defined in \eqref{eq:quadrature_conditioned},  
using \emph{control variates}  \cite{H1994,N2015}.
Define two mutually independent sets of random samples: 
$\widetilde{y}_1,\ldots,\widetilde{y}_{m}$ iid from $\mu$, and $y_1,\ldots,y_\M$ iid from $\sigma$. 
We consider the following cubature formula, that uses the above $2\M$ random samples as cubature nodes:  
\begin{equation}
\widehat{I}_{2m}(\fexp):= 
\widetilde \Q_\M(\fexp)
+ \dfrac{1}{m} \sum_{i=1}^{m} (\fexp - \widetilde{\fexp})  (\widetilde{y}_i). 
\label{eq:secondcub}
\end{equation}

The nodes $y_1,\ldots,y_\M,\widetilde{y}_1,\ldots,\widetilde{y}_\M$ are not identically distributed.  
The $\M$ random samples $y_1,\ldots,y_m$ are used to compute the weighted least-squares estimator $\widetilde{\fexp}$ of $\fexp$ and the cubature $\widetilde \Q_\M(\fexp)$ defined in \eqref{eq:quadrature_conditioned}, and then the $\M$ random samples $\widetilde{y}_1,\ldots,\widetilde{y}_{m}$ are used in the Monte Carlo estimator of $\fexp-\widetilde{\fexp}$. 
By construction $\widehat{I}_{2m}(\fexp) \approx I(\fexp)$ because 
$$
I(\fexp) - I(\widetilde{\fexp}) \approx \dfrac{1}{m} \sum_{i=1}^{m} (\fexp - \widetilde{\fexp}) (\widetilde{y}_i),  
$$
and $\widetilde \Q_\M(\fexp)=I(\widetilde{\fexp})$ for any $\fexp \in L^2(\Gamma,\mu)$, as a consequence of \eqref{eq:lemma_ex_phi}.   
The error of the cubature $\widehat{I}_{2m}(\fexp)$ satisfies the following theorem, whose proof is postponed to Section~\ref{sec:proofs}.
\begin{theorem}
\label{thm:estimates_quadrature_asy}
In any dimension $\s$, for any real $\param > 0$ and any $\delta \in(0,1)$, if 
 condition \eqref{eq:condition_points} holds true, 
and $\widetilde{y}_1,\ldots,\widetilde{y}_\M$ are i.i.d.~from $\mu$, 
and $\y_1,\ldots,\y_\M$ are i.i.d.~from $\sigma$ defined in \eqref{eq:prob_meas_aux}, 
and $\widetilde{y}_1,\ldots,\widetilde{y}_\M,\y_1,\ldots,\y_\M$ are mutually independent,  
then for any $\fexp \in L^2(\Gamma,\mu)$ it holds that 
\begin{equation}
\mathbb{E}\left( 
\left| I(\fexp) - \widehat{I}_{2\M}(\fexp)  \right|^2 
\right) \leq 
\dfrac{1}{m}
\left(
\left( 1+\varepsilon(m) \right) 
(e_2(\fexp))^2
+ 2 \| \fexp \|^2 m^{-r}
\right), 
\label{eq:estimation_error_asy}
\end{equation}
with $\varepsilon(\M)$ as in Theorem~\ref{thm:dls_multivariate}.
\end{theorem}

From the estimate in \eqref{eq:estimation_error_asy}, using Jensen inequality, we have that 
\begin{equation}
\mathbb{E}(  | I(\fexp)-   \widehat{I}_{2m}(\fexp)  | ) \leq  
\dfrac{1}{\sqrt{m}}
\left(
\sqrt{ 1+\varepsilon(m) } 
e_2(\fexp)
+ \sqrt2 \| \fexp \| m^{-r/2}
\right).  
\label{eq:estimate_new_jensen}
\end{equation}
In \eqref{eq:expectation_estimate_L1_alt} when using $2\M$ points we can choose an approximation space of dimension $n^*$ such that 
\begin{equation}
\label{eq:condition_points_star}
\dfrac{2\M}{\ln 2\M} \geq \dfrac{(1+r)}{\cdelta(\delta)} n^*. 
\end{equation}
We compare now the estimate \eqref{eq:expectation_estimate_L1_alt} for the cubature 
$\widetilde \Q_{2\M}(\fexp)$
on $V_{n^*}$ with $2\M$ points and the estimate \eqref{eq:estimate_new_jensen} for $\widehat{I}_{2\M}(\phi))$ on $V_n$ that also uses $2\M$ points. 
Suppose that $\min_{v\in V_n} \| \phi - v \|\leq C n^{-s}$ for some $s,C> 0$.  
For any $\M \geq 1$, $r>0$ and $\fexp \in L^2(\Gamma,\mu)$, it holds that $2|I(\fexp)| (2m)^{-r} \leq \sqrt{2} \| \fexp \| m^{-(r+1)/2}$. 
Both terms $m^{-(r+1)/2}$ and $(2m)^{-r}$ always decay sufficiently fast for $r$ large enough, 
\emph{i.e.}~$r \geq 2s-1$,  
and therefore,  
when comparing the convergence rates w.r.t.~$n$ of \eqref{eq:estimate_new_jensen} and \eqref{eq:expectation_estimate_L1_alt}, 
we can focus only on the term containing the best approximation error.  
Denote $C_\eqref{eq:estimate_new_jensen}:=\sqrt{ 1+\varepsilon(m) }$
and $C_\eqref{eq:expectation_estimate_L1_alt}:= 1+ \dfrac{\varepsilon(\M,n)}{1-\delta}$, 
that satisfy $C_\eqref{eq:estimate_new_jensen} \leq C_\eqref{eq:expectation_estimate_L1_alt}$  
for any $\delta \leq 1$, 
$n\geq 1$ and $m\geq 1$.  
From \eqref{eq:condition_points} and \eqref{eq:condition_points_star} $n^* < 2n$, 
and if $s\geq \frac12$ then 
$$
n > 2^{2s} 
\implies
C_{\eqref{eq:expectation_estimate_L1_alt}}
\sqrt{ \dfrac{n^*}{2\M} } (n^*)^{-s} 
>
C_{\eqref{eq:expectation_estimate_L1_alt}}
\sqrt{ \dfrac{2n}{2\M} } (2n)^{-s} 
> 
\dfrac{
C_{\eqref{eq:estimate_new_jensen}}
}{\sqrt{\M}} n^{-s}.  
$$

On the one hand, 
this shows that the 
error estimate \eqref{eq:estimate_new_jensen} 
has a better convergence rate w.r.t.~$n$  
than \eqref{eq:expectation_estimate_L1_alt}
when $n\geq 2^{2s}$, $s\geq \frac12$ and $r \geq 2s-1$. 
On the other hand, the estimate \eqref{eq:expectation_estimate_L1_alt} 
gives a better upper bound for the error 
when $n$ falls in the preasymptotic range, in particular when $s$ is large or $r < 2s-1$, 
and the bound has a better dependence on $\fexp$ since the term $|I(\fexp)|$ can be much smaller than $\| \fexp \|$. 

The cubature \eqref{eq:quadrature_conditioned} is exact on $V_n$, but this property does not hold anymore for the cubature \eqref{eq:secondcub}. 

\begin{remark}[Convergence rates of stochastic cubatures]
\label{conv_rates_remark}
From \eqref{eq:estimate_new_jensen}, we obtain explicit convergence rates 
of the cubature $\widehat{I}_{2\M}(\phi)$,  
assuming an algebraic decay $n^{-s}$ for some $s>0$ of the $L^2(\Gamma,\mu)$ best approximation error. For any $\param >0$, $n\geq 1$ and any $m$ satisfying \eqref{eq:condition_points}, for any $\fexp \in L^2(\Gamma,\mu)$:   
$$
e_2(\fexp) \lesssim n^{-s} 
\implies 
\mathbb{E}
\left(  \left| I(\phi)  - \widehat{I}_{2m}(\phi) \right|   \right)
\lesssim
\dfrac{1}{\sqrt m}
\left( n^{-s} + m^{-\param/2}
\right),
$$
where $\widehat{I}_{2\M}(\phi)$
uses $2\M$ evaluations of the function $\fexp$.
The above convergence rate can be made explicit w.r.t.~$\M$ as 
$$
\dfrac{1}{\sqrt m}
\left( n^{-s} + m^{-\param/2}
\right)
\lesssim m^{-s-1/2} (\log m)^{s} (1+\param)^s + m^{-\param/2-1/2}.   
$$
Taking $\param=s$ yields the convergence rate $m^{-s-1/2} (\log m)^{s}$ up to a constant independent of $m$ and $n$. 

Since condition \eqref{eq:condition_points} implies $\M \gtrsim 
(1+\param) n \log n$, it can also be rewritten w.r.t.~$n$ as    
$$
\dfrac{1}{\sqrt m}
\left( n^{-s} + m^{-\param/2}
\right) 
\lesssim 
n^{-s-1/2} \left( \dfrac{  \log n }{1+\param} \right)^{-1/2} 
+ n^{-\param/2-1/2} \left( \dfrac{  \log n }{1+\param} \right)^{-\param/2-1/2}
%
$$
and taking $\param=s$ the leading term is $n^{-s-1/2} (\log n)^{-1/2} $ up to a constant independent of $m$ and $n$. 
\end{remark}

\subsection{Randomized cubatures with strictly positive weights}
\noindent
In this section we construct cubature formulae of the form \eqref{eq:def_quadrature_general} with the weights \eqref{eq:def_weight_least_squares}, enforcing the additional property that all the weights $\quadweights_1,\ldots,\quadweights_\M$ are strictly positive. 
More precisely we prove that, 
if $\M$ is sufficiently larger than $n$, then 
$\quadweights_\indmeas \geq w(\y_\indmeas)/2\M$ for all $\indmeas=1,\ldots,\M$ with high probability.  
Strict positivity of $w(\y)$ for any $\y\in \Gamma$ 
therefore implies that the weights $\quadweights_1,\ldots,\quadweights_\M$ are all strictly positive. 
Define 
$$
w_\textrm{min}:=
\min_{\y \in \Gamma} w(\y)\leq  
\min_{\indmeas=1,\ldots,\M} w(\y_\indmeas)= 
n
\left(
\max_{ i=1,\ldots,\M } \sum_{j=1}^n |\basisfuncexp_j(\y_i)|^2
\right)^{-1}, 
$$ 
which is independent of $\M$ and depends on $V_n$. 
Of course $w_{\textrm{min}}\leq 1$. 
The next theorem establishes how much the cubature weights deviate from the weights of importance sampling in the nonasymptotic regime, when $w_{\textrm{min}}> 0$. 
The proof of this theorem is postponed to Section~\ref{sec:proofs}. 

\begin{theorem}
\label{thm:positive_weights}
In any dimension $\s$, for any real $\param > 0$ and any $n\geq 1$, if the $\M$ i.i.d.~nodes $\y_1,\ldots,\y_\M$ are drawn from $\sigma$ defined in \eqref{eq:prob_meas_aux} and $\M$ sastisfies 
\begin{equation}
\dfrac{\M}{\ln \M} \geq 
\dfrac{ 3(1 + \param)n^2  
 }{ 
(4\ln(4/3)-1)
w_{\textrm{min}} }
\label{eq:condition_points_pw}
\end{equation}
then
\begin{itemize}
\item[i)] the weights $\quadweights_1,\ldots,\quadweights_\M$ given by \eqref{eq:def_weight_least_squares} satisfy 
\begin{equation}
\Pr\left( \bigcap_{\indmeas=1}^{\M} \left\{ 
0<  
\dfrac{ 2 w(\y_\indmeas)- w_{\textrm{min}} }{2\M}  
\leq \quadweights_\indmeas 
\leq 
\dfrac{2w(\y_\indmeas) + w_{\textrm{min}}}{2\M} 
\right\}   \right) > 1-  2 \M^{-\param};
\label{eq:thesis_positive_weights_th}
\end{equation}
\item[ii)] the cubature formula \eqref{eq:def_quadrature_general} with weights \eqref{eq:def_weight_least_squares} satisfies items i), ii) and iii) of Theorem~\ref{thm:estimates_quadrature} with $\varepsilon(\M)$
replaced by 
\begin{equation}
\varepsilon(\M)
= \dfrac{4 \cdelta \left( \dfrac{\sqrt{ w_{\textrm{min}} } }{ 3\sqrt{n} }    \right)}{(1+\param)  \ln\M} \leq \dfrac{2 }{9(1+\param)  \ln\M}, 
\label{eq:upper_bound_eps_theo_wmin}
\end{equation}
in \eqref{eq:expectation_estimate}, 
and with $1/(1-\delta)$ replaced by $3/(2-2\delta)$ in \eqref{eq:expectation_estimate_L1_alt}. 
\end{itemize}
\end{theorem}

\begin{remark}
\label{thm:rmk_one_side_eq}
Since the following trivial inclusions 
between sets of random events 
hold true
\[
\bigcap_{\indmeas=1}^{\M} \left\{ \left\vert \quadweights_\indmeas - \dfrac{w(\y_\indmeas)}{\M} \right\vert \leq \dfrac{w_{\textrm{min}}}{2\M} \right\} 
\subseteq
\bigcap_{\indmeas=1}^{\M} \left\{ \left\vert \quadweights_\indmeas - \dfrac{w(\y_\indmeas)}{\M} \right\vert \leq \dfrac{w(\y_\indmeas)}{2\M} \right\} 
\subset 
\bigcap_{\indmeas=1}^{\M} \left\{   \quadweights_\indmeas \geq \dfrac{w(\y_\indmeas)}{2\M} 
\right\}
\subset 
\bigcap_{\indmeas=1}^{\M} \left\{   \quadweights_\indmeas 
> 0 \right\}
,  
\]
from 
\eqref{eq:thesis_positive_weights_th}
the event in the above right-hand side holds true with an even larger probability than $1-2\M^{-\param}$.
Therefore, since from \eqref{eq:thesis_positive_weights_th} all the weights $\quadweights_1,\ldots,\quadweights_\M$ are sandwiched between two strictly positive bounds, they are just strictly positive with an even larger probability than $1-2\M^{-\param}$. 
\end{remark}

It is worth to notice that 
the weights of the cubature \eqref{eq:secondcub} are not strictly positive.
Using the expression of  
$\widetilde{\fexp}$,  
the cubature \eqref{eq:secondcub} can be written 
in the form \eqref{eq:def_quadrature_general} 
as 
$$
\widehat{I}_{2\M}(\fexp) 
= 
\sum_{i=1}^\M 
\Big( 
\alpha_i  
\fexp(\widetilde{y}_i)
+
\widetilde{\alpha}_i 
\phi(y_i)
\Big)
$$
with nodes $y_1,\ldots,y_\M,\widetilde{y}_1,\ldots,\widetilde{y}_\M$ and 
 weights 
$$
\alpha_i= \dfrac{1}{\M}, \qquad 
\widetilde{\alpha}_i
=
\begin{cases}
- \dfrac{1}{\M} 
\left( 
\sum_{j=2}^n 
\left(
\dfrac{1}{\M}
\sum_{\ell=1}^\M  
\psi_j(\widetilde{y}_\ell) 
\right)
W^{1/2} D^\top G^{-1}   e_j  
\right)_i,  &  \textrm{if } \vvvert G-I \vvvert \leq \delta, \\
0, & \textrm{otherwise},
\end{cases}
\qquad 
i=1,\ldots,\M.
$$
From above, the weights $\widetilde{\alpha}_i$ might be negative. 
Proceeding as in the proof of \eqref{eq:thesis_positive_weights_th}, 
one can obtain conditions on $\M$ ensuring that with large probability  
$\sum_{i=1}^\M |\widetilde{\alpha}_i | \lesssim  \log \M $, 
see Remark~\ref{sec:rem_weights_sec}, which is a classical way to quantify the stability of a cubature in presence of negative weights.

\section{Multivariate polynomial approximation spaces}
\label{sec:four}
\noindent
In this section we assume that the domain $\Gamma\subseteq \mathbb{R}^\s$ has a Cartesian product structure, 
\begin{equation}
\Gamma:= \times_{\inddim=1}^\s \Gamma_\inddim,
\label{eq:ass_cartesian}
\end{equation}
where $\Gamma_\inddim\subseteq \mathbb{R}$ are bounded or unbounded intervals. 
We further assume that $d\mu = \otimes_{\inddim=1}^\s d\mu_\inddim$, where each $\mu_\inddim$ is a probability measure on $\Gamma_\inddim$.  
For convenience we take $\Gamma_1=\Gamma_\inddim$ and $\mu_1=\mu_\inddim$ for any $\inddim=2,\ldots,\s$. 
Assume the existence of a familiy $(\varphi_\indbasis)_{\indbasis \geq 0}$ of univariate orthogonal polynomials complete and orthonormal in $L^2(\Gamma_1,\mu_1)$. 
For any $\nu \in \Nset_0^\s$ we define the multivariate polynomials
\begin{equation}
\basisfuncexp_\nu(\y):= \prod_{\inddim=1}^\s \varphi_{\nu_\inddim}(\y^{(\inddim)}), \quad \y=(\y^{(1)},\ldots,\y^{(\s)}) \in \Gamma,
\label{eq:multi_basis_pol}
\end{equation}
with $\y^{(\inddim)}$ being the $\inddim$th coordinate of $\y$. 
The set $(\basisfuncexp_\nu)_{\nu \in \mathbb{N}_0^\s}$ is a complete orthonormal basis in $L^2(\Gamma,\mu)$. 

Consider any finite $\s$-dimensional multi-index set $\Lambda \subset \Nset^\s_0$, and denote its cardinality by $\#(\Lambda)$.  
We denote the polynomial space $\Pol_\Lambda=\Pol_\Lambda(\Gamma)$ associated with $\Lambda$ as  
\begin{equation}
\Pol_\Lambda:=\textrm{span}\left\{\basisfuncexp_\nu \ : \ \nu \in \Lambda \right\}. 
\label{eq:def_pol_space}
\end{equation}
The result from the previous sections apply to the polynomial setting by taking $V_n=\Pol_{\Lambda}$ with $n=\#(\Lambda)$. 
A remarkable class of index sets are downward closed index sets. 
\begin{definition}[Downward closed multi-index set $\Lambda$]
\label{thm:def_downward_closed}
In any dimension $\s$, a finite multi-index set $\Lambda\subset \mathbb{N}_0^\s$ is downward closed, 
if 
\[
\nu \in \Lambda \implies \widetilde \nu \leq \nu, \quad \forall \  \widetilde\nu \in \Lambda, 
\]
where $\widetilde\nu \leq \nu$ is meant component-wise, \emph{i.e.}~$\widetilde\nu_\inddim \leq \nu_\inddim$ for any $\inddim=1,\ldots,\s$.
\end{definition}

A relevant setting in which this type of index sets arises is Gaussian integration, where $\mu$ is the Gaussian measure on $\Gamma=\mathbb{R}^\s$, and the $\basisfuncexp_\nu$ are tensorized Hermite polynomials. 
On the one hand, tensorization of univariate Gaussian quadratures becomes prohibitive as the dimension $\s$ increases, due to the exponential growth in $\s$ of the number of nodes. 
On the other hand, the use of downward closed sets allows one to tune the polynomial space and allocate only the most effective degrees of freedom, depending on the importance of each coordinate in the approximation of the target function.

\begin{remark}[Inverse inequalities for polynomials supported on downward closed index sets]
\label{thm:remark_inv_ineq}
Given two integer parameters $(\parajac_1,\parajac_2) \in \mathbb{N}_0 \times \mathbb{N}_0 \cup \{(-1/2,-1/2)\} $, let $\mu$ be the probability measure on $\Gamma=[-1,1]^\s$ defined as 
$d\mu =  \otimes
^\s  dJ$ with 
$$
dJ:=C(1-t)^{\parajac_1}  (1+t)^{\parajac_2} \, d\lambda(t), \quad t\in [-1,1], 
\textrm{ and } C \textrm{ s.t. } \int_{-1}^{+1} dJ(t) =1. 
$$
Consider the tensorized Jacobi polynomials $(J^{\parajac_1,\parajac_2}_\nu)_{\nu\in \mathbb{N}_0^\s}$ 
constructed by \eqref{eq:multi_basis_pol} when taking 
$(\varphi_k)_{k\geq 0}$ as the sequence of univariate Jacobi polynomials orthonormal in $L^2([-1,1],dJ)$. 
The $(J^{\parajac_1,\parajac_2}_\nu)_{\nu}$ corresponds to 
tensorized Legendre polynomials when $\parajac_1=\parajac_2=0$,
and to tensorized Chebyshev polynomials when  $\parajac_1=\parajac_2=-\frac12$. 
Choosing the orthonormal basis $(\basisfuncexp_\nu)_\nu$ as $(J^{\parajac_1,\parajac_2}_\nu)_{\nu
}$ and given any downward closed index set $\Lambda\subset \mathbb{N}_0^\s$ with cardinality equal to $n$, see Definition \ref{thm:def_downward_closed},
define $V_n$ using \eqref{eq:def_pol_space} as the space generated by $(J^{\parajac_1,\parajac_2}_\nu)_{\nu \in \Lambda}$. 
In the setting described above, the following inverse inequalities are proven in \cite{CCMNT2013,M2015}: 
for all $v\in V_n$ it holds that 
\begin{equation}
\| v  \|_{L^\infty} \leq n^{B(\parajac_1,\parajac_2)} \| v \|, 
\label{eq:inverse_ineq_pol}
\end{equation}
where
\begin{equation}
B(\parajac_1,\parajac_2)
:= 
\begin{cases}
\max\{\parajac_1,\parajac_2  \}+1, & (\parajac_1,\parajac_2) \in \mathbb{N}_0 \times \mathbb{N}_0, \\
\dfrac{ \ln 3}{2\ln 2}, & (\parajac_1,\parajac_2)=\left(-\frac12,-\frac12\right), 
\end{cases}
\label{eq:def_cost_inv_ineq}
\end{equation}
One step that leads to the proof of such inequalities, see~\cite{CCMNT2013,M2015}, is the estimate 
\begin{equation}
\label{eq:max_bound_inv_ineq}
\max_{\y \in \Gamma } \sum_{k=1}^n |\basisfuncexp_k(\y)|^2 \leq n^{2B(\parajac_1,\parajac_2)}. 
\end{equation}
\end{remark}

\begin{remark}[Lower bound on $w$ for polynomial spaces]
\label{sec:lower_bound_polynomials}
From a numerical standpoint, it is desirable that the weights are not only strictly positive but also bounded away from zero. 
When $V_n$ is  a polynomial space on $[-1,1]^\s$
generated by the $(J^{\parajac_1,\parajac_2}_\nu)_{\nu \in \Lambda}$ with $\Lambda$ downward closed, strictly positive lower bounds for the weights can be derived by using 
the estimate \eqref{eq:max_bound_inv_ineq}. 
Using 
\eqref{eq:max_bound_inv_ineq}
we obtain the following lower bound  
uniformly over $\Gamma$ for $w$: 
\begin{equation}
w(\y)
\geq 
n^{1-2B(\parajac_1,\parajac_2)}, 
\quad \y \in \Gamma,    
\label{eq:lower_bound_w}
\end{equation}
with $\parajac_1,\parajac_2$ being the same parameters that appear in Remark~\ref{thm:remark_inv_ineq}. 
\end{remark}

The next result is a corollary of Theorem~\ref{thm:positive_weights} in the particular case that $V_n$ is a polynomial space, and are obtained by using the lower and upper bounds   \eqref{eq:lower_bound_w} and \eqref{eq:upper_bound_w}.  

\begin{corollary}
\label{thm:positive_weights_coro}
In any dimension $\s$ with $\Gamma=[-1,1]^\s$, let $\mu$ be the Jacobi probability measure on $\Gamma$ and $V_n$ be any downward closed polynomial space generated by tensorized Jacobi polynomials.   
For any real $\param > 0$ and $n\geq 1$, 
if the $\M$ i.i.d.~nodes $\y_1,\ldots,\y_\M$ are drawn from $\sigma$ defined in \eqref{eq:prob_meas_aux} and $\M$ sastisfies 
\begin{equation}
\dfrac{\M}{\ln \M} \geq 
\dfrac{ 3 (1 + \param)  
 }{ 
4\ln(4/3)-1
} n^{2B(\parajac_1,\parajac_2)+1} , \quad \left( \parajac_1,\parajac_2 \right) \in \mathbb{N}_0\times\mathbb{N}_0 \cup \left\{ \left(- \frac12, -\frac12 \right) \right\}, 
\label{eq:condition_points_pw_coro_weight}
\end{equation}
then 
\begin{itemize}
\item[i)] the weights $\quadweights_1,\ldots,\quadweights_\M$ given by \eqref{eq:def_weight_least_squares} satisfy 
\begin{equation*}
\Pr\left( \bigcap_{\indmeas=1}^{\M} 
\left\{ 
\dfrac{1}{2 \M n^{B(\parajac_1,\parajac_2)}} 
\leq  
\quadweights_\indmeas  
\leq 
\dfrac{3  \sqrt{n} }{2 \M}
\right\}   \right) 
> 1-  2 \M^{-\param}.
\end{equation*}
\item[ii)]
the cubature formula \eqref{eq:def_quadrature_general} with weights \eqref{eq:def_weight_least_squares} satisfies items i), ii) and iii) of Theorem~\ref{thm:estimates_quadrature} with $\varepsilon(\M)$
replaced by 
\begin{equation*}
\varepsilon(\M)
= \dfrac{4 \cdelta \left( \dfrac{\sqrt{ w_{\textrm{min}} } }{ 3\sqrt{n} }    \right)}{(1+\param)  \ln\M} \leq \dfrac{2 }{9 
(1+\param)  \ln\M}, 
\end{equation*}
in \eqref{eq:expectation_estimate}, 
and with $1/(1-\delta)$ replaced by $3/(2-2\delta)$ in \eqref{eq:expectation_estimate_L1_alt}. 
\end{itemize}

\end{corollary}

The next corollary contains a similar result as Corollary~\ref{thm:positive_weights_coro} but choosing $w\equiv 1$, that corresponds to using standard least squares with random samples distributed as $\mu$.

\begin{corollary}
\label{thm:positive_nonweights_coro}
In any dimension $\s$ with $\Gamma=[-1,1]^\s$, let $\mu$ be the Jacobi probability measure on $\Gamma$, and $V_n$ be any downward closed polynomial space generated by tensorized Jacobi polynomials.  
For any real $\param > 0$ and $n\geq 1$,  
if the $\M$ i.i.d.~nodes $\y_1,\ldots,\y_\M$ are drawn from $\mu$ and $\M$ sastisfies 
\begin{equation}
\dfrac{\M}{\ln \M} \geq 
\dfrac{ 3 (1 + \param)  
 }{ 
4\ln(4/3)-1
} n^{4B(\parajac_1,\parajac_2)}, \quad \left( \parajac_1,\parajac_2 \right) \in \mathbb{N}_0\times\mathbb{N}_0 \cup \left\{ \left( -\frac12,-\frac12 \right) \right\}, 
\label{eq:condition_points_pw_coro_unweight}
\end{equation}
then 
\begin{itemize}
\item[i)] the weights $\quadweights_1,\ldots,\quadweights_\M$ given by \eqref{eq:def_weight_least_squares} satisfy 
\begin{equation*}
\Pr\left( \bigcap_{\indmeas=1}^{\M} 
\left\{ 
\dfrac{1}{2 \M } 
\leq  
\quadweights_\indmeas  
\leq 
\dfrac{3  }{2 \M}
\right\}   \right) 
> 1-  2 \M^{-\param}.
\end{equation*}
\item[ii)]
the cubature formula \eqref{eq:def_quadrature_general} with weights \eqref{eq:def_weight_least_squares} satisfies items i), ii) and iii) of Theorem~\ref{thm:estimates_quadrature} with $\varepsilon(\M)$ replaced by 
\begin{equation*}
\varepsilon(\M)= \dfrac{4 \cdelta \left( \dfrac{1}{ 3 n^{B(\parajac_1,\parajac_2)} }    \right)}{(1+\param)  \ln\M} \leq \dfrac{2 }{9n^{2B(\parajac_1,\parajac_2)} (1+\param)  \ln\M} 
\end{equation*}
in \eqref{eq:expectation_estimate}, 
and with $1/(1-\delta)$ replaced by $3/(2-2\delta)$ in \eqref{eq:expectation_estimate_L1_alt}. 

\end{itemize}

\end{corollary}

The exponent of $n$ in condition \eqref{eq:condition_points_pw_coro_weight} with weighted least squares is always smaller than that in condition \eqref{eq:condition_points_pw_coro_unweight} with standard least squares.  
In the Legendre and Chebyshev cases, Corollary~\ref{thm:positive_weights_coro} ensures positivity of the weights with large probability if  
$$
\dfrac{\M}{\ln \M} \geq \dfrac{ 3(1+\param) }{ 4\ln(4/3)-1 }  n^{s}, 
$$
with 
\begin{align*}
s & = 2B(0,0) + 1 =3,   
\qquad \qquad \qquad \qquad \qquad \quad \ \,  
\textrm{ in the Legendre case, and} \\
s & = 2B\left(-\frac12,-\frac12\right) + 
1
= 
\frac{ \ln 3 }{\ln 2} +1 
 \approx 2.585, 
\qquad 
\textrm{ in the Chebyshev case}.  
\end{align*}

\section{Proofs and intermediate results}
\label{sec:proofs}
\noindent
In this section we use the notation $d\mu_\M:=\otimes^\M d\mu$. 

\begin{proof}[Proof of Lemma~\ref{thm:lemma_weights_exactness}]
Proof of \eqref{eq:lemma_ex_phi}. 
On the one hand, using in sequence $\basisfuncexp_1 \equiv 1$, the orthogonality property of the basis functions and $\| \basisfuncexp_1 \|=1$, we have that 
\begin{align}
I(\Pi_n^\M \fexp) = & \int_{\Gamma} \Pi_n^\M \fexp \, d\mu \\
= & \int_{\Gamma} \sum_{j=1}^n  \coef_j \basisfuncexp_j \,  d\mu \nonumber \\
= & \coef_1 ,
\end{align}
with $\coef_1$ being the coefficient associated to $\basisfuncexp_1$ in the expansion \eqref{eq:expansion_pol_space}. 

On the other hand, the left-hand side in \eqref{eq:lemma_ex_phi} is the cubature formula \eqref{eq:def_quadrature_general}, 
that can be read (up to a multiplicative factor $\M^{-1}$) as the scalar product in $\Rset^\M$ between the vector $\quadweights$
containing the weights and the vector $\rhsfun$ containing the evaluations of the function $\fexp$ at the nodes, 
and from \eqref{eq:coefficients} we have 
\begin{equation}
\Q_\M(\fexp)
= \dfrac{1}{\M} \quadweights^\top \rhsfun = 
\dfrac{1}{\M}  \left(  \design \gramian^{-1} e_1  \right)^\top \rhs
=
\dfrac{1}{\M}  e_1^\top \gramian^{-1} \design^\top \rhs
=
e_1^\top \coef
=  
\coef_1, 
\label{eq:first_coefficient}
\end{equation} 
that proves the equality \eqref{eq:lemma_ex_phi}. 

Proof of \eqref{eq:lemma_ex_pol}. We notice that, since $\Pi_n^\M$ is a projection on 
$V_n$, then it holds $\Pi_n^\M \genpol \equiv \genpol$ for any $\genpol\in V_n$, and we obtain \eqref{eq:lemma_ex_pol} from \eqref{eq:lemma_ex_phi}.
\end{proof}

\begin{proof}[Proof of Theorem~\ref{thm:estimates_quadrature}]
The proof of i) is an immediate consequence of Lemma~\ref{thm:lemma_weights_exactness} and Corollary~\ref{thm:coro_well_conditioned}. 

For proving ii), using point i) we bound the integration error as 
\begin{align}
\left|  
I(\fexp)
-
\Q_{\M}(\fexp) 
  \right| 
= &  
\left|
\int_\Gamma 
\left( 
\fexp 
-
\Pi_n^m \fexp
\right) 
d\mu
\right| 
\nonumber
\\ 
\leq & 
\int_\Gamma 
\left| 
\fexp 
-
\Pi_n^m \fexp
\right| 
d\mu 
\nonumber
\\
\leq  &
\|
\fexp
-
\Pi_n^m \fexp 
 \|, 
\label{eq:up_bound_error_int}
\end{align}
and combining with \eqref{eq:prob_estimate} we obtain \eqref{eq:probabilistic_estimate}. 

Proof of iii) estimate \eqref{eq:expectation_estimate}. 
We start by splitting the expectation in \eqref{eq:expectation_estimate} over the sets of events $\{ \vvvert \gramian - \cI \vvvert \leq \delta \} $ and $\{ \vvvert \gramian - \cI \vvvert > \delta \}$.  
Since on the event $\{ \vvvert \gramian - \cI \vvvert \leq \delta \} $ the cubature $\widetilde \Q_{\M}(\fexp)$ equals $\Q_{\M}(\fexp) $, we obtain
$$
\E\left(
\left|  
I(\fexp)
-
\widetilde \Q_{\M}(\fexp) 
  \right|^2 
\right)
=
\int_{ \vvvert \gramian - \cI \vvvert \leq \delta  } 
\left|  
I(\fexp)
-
\Q_{\M}(\fexp)
  \right|^2 
\, d\mu_\M
+
\int_{ \vvvert \gramian - \cI \vvvert > \delta  } 
\left|  
I(\fexp)
-
\widetilde \Q_{\M}(\fexp) 
  \right|^2 
\, d\mu_\M. 
$$
Then we use the upper bound \eqref{eq:up_bound_error_int} and proceeding as in the proof of \eqref{eq:expect_estimate} in \cite{CM2016} we obtain 
\begin{equation}
\int_{ \vvvert \gramian - \cI \vvvert \leq \delta  } 
\left|  
I(\fexp)
-
\Q_{\M}(\fexp)
  \right|^2 
\, d\mu_\M
\leq 
\int_{ \vvvert \gramian - \cI \vvvert \leq \delta  } 
\|
\fexp
-
\Pi_n^m \fexp 
 \|^2
\, d\mu_\M
\leq 
(1+\varepsilon(\M))(e_2(\fexp))^2.
\label{eq:brutal_upper_bound}
\end{equation}

The last term in the right-hand side of 
\eqref{eq:expectation_estimate}
is an upper bound for  
 the integral on the event $\vvvert \gramian - \cI \vvvert > \delta $, where $\widetilde \Q_{\M}(\fexp)$ is set to zero.

Proof of iii) estimate \eqref{eq:expectation_estimate_L1_alt}. 
Splitting the expectation in \eqref{eq:expectation_estimate_L1_alt} over the events $\{ \vvvert \gramian - \cI \vvvert \leq \delta \}$ and $\{ \vvvert \gramian - \cI \vvvert > \delta \}$, 
using \eqref{eq:quadrature_conditioned} and 
\eqref{eq:bias_var_Q_m_mat_nodec}
we obtain 
\begin{align}
\E\left( \left|  I(\fexp) - \widetilde \Q_\M(\fexp) \right| \right)  
= & 
\underbrace{
\int_{ \vvvert \gramian - \cI  \vvvert \leq \delta } \left| \dfrac{1}{\M}  e_1^\top \gramian^{-1} \design^\top W^{1/2} g  \right| \, d\mu_\M
}_{=:A}
+  \underbrace{\int_{ \vvvert \gramian - \cI  \vvvert > \delta } \left| I(\fexp) \right| \, d\mu_\M}_{=:B}. 
\label{eq:error_splitting_proof_est_L1}
\end{align}
Term B 
can be controlled as 
\begin{equation}
\label{eq:unstable_term}
  \int_{ \vvvert \gramian - \cI  \vvvert > \delta } \left| I(\fexp) \right| \, d\mu_\M
\leq  2 
| I(\fexp) | \M^{-\param}. 
\end{equation}

For term A, 
using \eqref{eq:bias_var_Q_m_mat},  triangular inequality,
the sub-multiplicative property of the operator norm,  
 \eqref{eq:inverse_gramian}
and Cauchy-Schwarz inequality we obtain 
\begin{align}
A & \leq \int_{ \vvvert \gramian - \cI  \vvvert \leq \delta } \left| \dfrac{1}{\M}  e_1^\top  \design^\top W^{1/2} g  \right| \, d\mu_\M +
\int_{ \vvvert \gramian - \cI  \vvvert \leq \delta } \left| \dfrac{1}{\M} e_1^\top (\gramian^{-1}-\cI) \design^\top W^{1/2} g  \right|
\, 
d \mu_\M
\nonumber \\
& \leq 
\E\left(  \left| \dfrac{1}{\M}  e_1^\top  \design^\top W^{1/2} g  \right| \right) +
\int_{ \vvvert \gramian - \cI  \vvvert \leq \delta } 
\vvvert \gramian^{-1} - \cI \vvvert  
\left\| \dfrac{1}{\M} \design^\top W^{1/2} g \right\|_{\ell_2}
\, 
d \mu_\M
\nonumber \\
& \leq 
\left( \E\left(  \left| \dfrac{1}{\M} e_1^\top  \design^\top W^{1/2} g  \right|^2 \right)\right)^{1/2} +
\dfrac{1}{1-\delta}
\int_{ \vvvert \gramian - \cI  \vvvert \leq \delta } 
\vvvert \gramian - \cI  \vvvert
\left\| \dfrac{1}{\M} \design^\top W^{1/2} g \right\|_{\ell_2} 
\, 
d \mu_\M
\nonumber \\
& \leq 
\left( \E\left(  \left| \dfrac{1}{\M} e_1^\top  \design^\top W^{1/2} g  \right|^2 \right)\right)^{1/2} +
\dfrac{1}{1-\delta}
\left(
\E \left(  
\vvvert \gramian - \cI  \vvvert^2
\right) 
\right)^{1/2} 
\left( 
\E \left( 
\left\| \dfrac{1}{\M} \design^\top W^{1/2} g \right\|_{\ell_2}^2 
\right)
\right)^{1/2}.  
\label{eq:bound_A}
\end{align}
Using Lemmas~\ref{thm:bound_exp_bernstein}, \ref{thm:bound_exp_abs_value} and \ref{thm:bound_exp_l2_norm} to bound the expectations in  \eqref{eq:bound_A}  
 we obtain \eqref{eq:expectation_estimate_L1_alt}.
\end{proof}

One of the results used in the proof of \eqref{eq:expectation_estimate_L1_alt} is the following upper bound on the spectral norm of sum of independent random matrices, see for example \cite[Theorem 4.1]{Tropp2015}, that we rewrite here 
in the Hermitian case. 
We denote by $0_n$ the null $n$-by-$n$ matrix.

\begin{theorem}
\label{thm:bernstein}
Consider a family $(\matber^\indmeas)_{1\leq \indmeas \leq \M }$ of independent random matrices in $\mathbb{R}^{n\times n}$  such that  $\E(\matber^\indmeas)=0_n$ for all $\indmeas$, and 
define $\matsum=\sum_{\indmeas=1}^\M \matber^\indmeas$. 
Then 
$$
\left( 
\E(\vvvert \matsum \vvvert^2 )
\right)^{1/2}
 \leq \sqrt{ C(n)  } 
\vvvert \E( \matsum^\top Z) \vvvert^{1/2} 
+ C(n) \left(  \E \left( \max_{\indmeas=1,\ldots,\M} \vvvert \matber^\indmeas \vvvert^2 \right) \right)^{1/2}, 
$$
with 
$$
C(n):=4(1+2\lceil \ln n \rceil ). 
$$
\end{theorem}

\begin{lemma}
\label{thm:bound_exp_bernstein}
\begin{equation}
\left(
\E \left( 
\vvvert \gramian - \cI  \vvvert^2
\right) \right)^{1/2} \leq 
\sqrt{ 4(1+2 \lceil \ln n \rceil ) }
\sqrt{  
 \dfrac{n-1}{\M}}
\left( 
1+ \sqrt{ 4(1+2 \lceil \ln n \rceil ) } \sqrt{ \dfrac{n}{\M} }
\right).
\label{eq:bound_exp_bernstein}
\end{equation}
\end{lemma}
\begin{proof}
Using the random variable $\y$ distributed as $\sigma$, 
we introduce 
the $n$-by-$n$ real random matrices $\matsin=\matsin(\y)$ and $\matber=\matber(\y)$, whose elements are defined as  
$$
\matsin_{pq}(\y):=\dfrac{w(\y)}{\M} \basisfuncexp_p(\y) \basisfuncexp_q(\y),\quad \matber_{pq}(\y) := \matsin_{pq}(\y) - \dfrac{\delta_{pq}}{\M}, \quad p,q=1,\ldots,n.
$$
By construction, $\E(\matsin)=\frac1\M \cI$ and therefore $\E(\matber)=0_n$. 
Denote by $(\matsin^\indmeas)_{1\leq \indmeas \leq \M}$ a family of $\M$ independent copies of $\matsin$,  
and by $(\matber^\indmeas)_{1\leq \indmeas \leq \M}$ a family of $\M$ independent copies of $\matber$. 

The matrix $\matsin^\indmeas$ has rank one, and $\matber^\indmeas$ has full rank. Nonetheless, we can compute an upper bound for $\vvvert \matber^\indmeas \vvvert$ as follows: 
\begin{align}
\nonumber
\vvvert \matber^\indmeas \vvvert^2 & \leq  \vvvert \matber^\indmeas \vvvert_F^2 
\\
&= \textrm{trace}\left( (\matber^\indmeas)^\top \matber^\indmeas \right) 
\nonumber
\\
& = \textrm{trace}\left( ( \matsin^\indmeas - \M^{-1} \cI )^\top ( \matsin^\indmeas - \M^{-1} \cI )  \right)
\nonumber
\\
& = \sum_{p=1}^n \left( \sum_{q=1}^n \matsin_{pq}^\indmeas \matsin_{qp}^\indmeas - \dfrac{2\matsin_{pp}^\indmeas}{\M} + \dfrac{\delta_{pp}}{\M^2}   \right)
\nonumber
\\
& = \sum_{p=1}^n \sum_{q=1}^n \left( \matsin_{pq}^\indmeas\right)^2 
- \dfrac{2}{\M} \sum_{p=1}^n \matsin_{pp}^\indmeas + \dfrac{n}{\M^2}  
\nonumber
\\
& = \dfrac{|w(\y_\indmeas)|^2 }{\M^2} \sum_{p=1}^n \sum_{q=1}^n  |\basisfuncexp_{p}(\y_\indmeas)|^2 | \basisfuncexp_q(\y_\indmeas)|^2  
- \dfrac{2w(\y_\indmeas)}{\M^2} \sum_{p=1}^n 
| \basisfuncexp_{p}(\y_\indmeas)|^2 
+ \dfrac{n}{\M^2}  
\nonumber
\\
& = \dfrac{n\left( n-1 \right) }{\M^2}, 
\label{eq:unif_bound_spec_norm}
\end{align}
and the trace has been rewritten using 
$$
\left( ( \matsin^\indmeas - \M^{-1} \cI )^\top ( \matsin^\indmeas - \M^{-1} \cI ) \right)_{pq} = \sum_{j=1}^n 
( \matsin_{pj}^\indmeas - \M^{-1} \delta_{pj} )( \matsin_{jq}^\indmeas - \M^{-1} \delta_{jq} ) 
= \sum_{j=1}^n \matsin_{pj}^\indmeas X_{jq}^\indmeas - \dfrac{2}{\M} X_{pq}^\indmeas + \dfrac{ \delta_{pq} }{\M^2},  
$$
and 
\begin{equation}
\sum_{j=1}^n \delta_{pj} =\delta_{pp}, 
\qquad 
\sum_{j=1}^n \delta_{pj} \delta_{jq}=\delta_{pq}.
\label{eq:summation_prod_deltas}
\end{equation}
The bound \eqref{eq:unif_bound_spec_norm} holds uniformly for all $\indmeas=1,\ldots,\M$, and therefore 
\begin{equation}
\E\left( \max_{\indmeas=1,\ldots,\M} \vvvert \matber^\indmeas \vvvert^2  \right) \leq  \dfrac{n\left( n-1 \right) }{\M^2}.
\label{eq:bound_exp_ber}
\end{equation}

Define $\matsum:=\sum_{\indmeas=1}^\M \matber^\indmeas=\gramian-\cI$ and let us  compute $\E( \matsum^\top \matsum )$. The components of the matrix $\matsum$ can be written as $\matsum_{pq}=\langle \basisfuncexp_p,\basisfuncexp_q \rangle_\M - \delta_{pq}$, and therefore
\begin{align*}
(\matsum^\top \matsum)_{pq} 
= & \sum_{k=1}^n  \left( \langle \basisfuncexp_p,\basisfuncexp_k \rangle_\M - \delta_{pk}  \right)
\left( \langle \basisfuncexp_k,\basisfuncexp_q \rangle_\M - \delta_{kq}  \right) 
\\
= & \sum_{k=1}^n  \langle \basisfuncexp_p,\basisfuncexp_k \rangle_\M  \langle \basisfuncexp_k,\basisfuncexp_q \rangle_\M  - 2  \langle \basisfuncexp_p,\basisfuncexp_q \rangle_\M + \delta_{pq}, 
\end{align*}
where at the last step we have used \eqref{eq:summation_prod_deltas}.
Taking the expectation on both sides gives
$$
\E( (\matsum^\top \matsum)_{pq}  ) = \sum_{k=1}^n \underbrace{ 
\E(  \langle \basisfuncexp_p,\basisfuncexp_k \rangle_\M  \langle \basisfuncexp_k,\basisfuncexp_q \rangle_\M )
}_{=:T_{kpq}} - 
\delta_{pq} . 
$$
Using the independence of the random samples, the term $T_{kpq}$ can be rewritten as 
\begin{align*}
T_{kpq} = & \dfrac{1}{\M^2} \E\left(  \left( \sum_{\indmeas=1}^\M  w(\y_\indmeas) \basisfuncexp_p(\y_\indmeas) \basisfuncexp_k(\y_\indmeas) \right) \left( \sum_{j=1}^\M 
w(\y_j ) \basisfuncexp_k(\y_j ) \basisfuncexp_q(\y_j )
\right)   \right) 
\\
= &
\dfrac{1}{\M^2} 
\left( 
\E\left( \sum_{\indmeas=1}^\M  w(\y_\indmeas) \basisfuncexp_p(\y_\indmeas) \basisfuncexp_k(\y_\indmeas) 
 \sum_{j=1\atop j\neq \indmeas }^\M 
w(\y_j ) \basisfuncexp_q(\y_j ) \basisfuncexp_k(\y_j )
\right) 
+ 
\E\left( 
\sum_{\indmeas=1}^\M  (w(\y_\indmeas))^2 \basisfuncexp_p(\y_\indmeas) \basisfuncexp_q(\y_\indmeas) (\basisfuncexp_k(\y_\indmeas))^2 
\right) 
\right),  
\end{align*}
and using linearity of expectation, the definition of $w$ 
and \eqref{eq:summation_prod_deltas}
we obtain 
$$
\sum_{k=1}^n T_{kpq} =\dfrac{1}{\M^2} \left( \sum_{k=1}^n  \left( \M \delta_{pk} (\M-1) \delta_{qk}) \right)+ n\M \delta_{pq}    \right) = \dfrac{\M + n -1}{\M } \delta_{pq}. 
$$
Hence we have finally  
$$
\E( \matsum^\top \matsum)  ) =  \dfrac{ n -1}{\M } \cI, 
$$
and 
\begin{equation}
\vvvert \E( \matsum^\top \matsum)  ) \vvvert = \dfrac{(n-1)}{m}. 
\label{eq:bound_var_ber}
\end{equation}

We now apply Theorem~\ref{thm:bernstein} to the family of random matrices 
$\matber^1,\ldots,\matber^\M$, 
using the bounds \eqref{eq:bound_exp_ber} and \eqref{eq:bound_var_ber}, 
and finally obtain \eqref{eq:bound_exp_bernstein}. 
\end{proof}

\begin{lemma}
\label{thm:bound_exp_abs_value}
\begin{equation*}
\E\left(  \left| \dfrac{1}{\M} e_1^\top  \design^\top W^{1/2} g  \right|^2 \right)
\leq 
\dfrac{n}{\M} (e_2(\fexp))^2.
\end{equation*}
\end{lemma}
\begin{proof}
Since {\xc 
from \eqref{eq:zero_mean_err_w}
} $\E(wg)=0$, it holds that 
\begin{align}
\E\left(  \left| \dfrac{1}{\M} e_1^\top  \design^\top W^{1/2} g  \right|^2 \right) = &  
\E\left( \left( \dfrac{1}{\M} \sum_{\indmeas=1}^\M w(\y_\indmeas) \err(\y_\indmeas) \right)^2 \right) 
\nonumber
\\
= & 
\dfrac{1}{\M^2}
 \sum_{\indmeas,j=1}^\M
\E\left(   w(\y_\indmeas) \err(\y_\indmeas) w(\y_j) \err(\y_j)  \right) 
\nonumber
\\
= & 
\dfrac{1}{\M^2}
 \sum_{\indmeas=1}^\M
\E\left(  ( w(\y_\indmeas) \err(\y_\indmeas) )^2  \right) 
\nonumber
\\
= & 
\dfrac{1}{\M}
\int_\Gamma w(\y) (\fexp(\y) - \Pi_n \fexp(\y) )^2  \, d\mu 
\label{eq:error_similar_to_imp_samp}
\\
\leq & 
\nonumber
\dfrac{n}{\M}
(e_{2}(\fexp))^2, 
\end{align}
and at the last step we have used \eqref{eq:upper_bound_w}. 
\end{proof}

\begin{lemma}
\label{thm:bound_exp_l2_norm}
\begin{equation*}
\E \left( 
\left\| \dfrac{1}{\M} \design^\top W^{1/2} g \right\|_{\ell_2}^2 
\right)
=
\dfrac{n}{\M} (e_2(\fexp))^2. 
\end{equation*}
\end{lemma}
\begin{proof}
Using the independence of the random samples 
\begin{align*}
\E \left( 
\left\| \dfrac{1}{\M} \design^\top W^{1/2} g \right\|_{\ell_2}^2 
\right) & = 
\dfrac{1}{\M^2} \E \left( \sum_{k=1}^n \left( 
\sum_{\indmeas=1}^\M  
  \design_{\indmeas k} \sqrt{w(\y_\indmeas)} \err(\y_\indmeas) 
\right)^2 \right)
\\
& = 
\dfrac{1}{\M^2}  \sum_{k=1}^n 
\sum_{\indmeas,j=1}^\M  
\E \left(  
  \basisfuncexp_k(\y_\indmeas) w(\y_\indmeas) \err(\y_\indmeas)
\basisfuncexp_k(\y_j) w(\y_j) \err(\y_j) 
\right)
\\
& = 
\dfrac{1}{\M^2}  \sum_{k=1}^n
\left(
\sum_{\indmeas=1}^\M  
\E \left(  
  (\basisfuncexp_k(\y_\indmeas) w(\y_\indmeas) \err(\y_\indmeas))^2
\right)+
\sum_{\indmeas,j=1 \atop \indmeas \neq j}^\M  
\E \left(  
  \basisfuncexp_k(\y_\indmeas) w(\y_\indmeas) \err(\y_\indmeas)
\basisfuncexp_k(\y_j) w(\y_j) \err(\y_j) 
\right)
\right)
\\
& = 
\dfrac{1}{\M^2}   
\sum_{\indmeas=1}^\M  
\E \left(  
(w(\y_\indmeas) \err(\y_\indmeas))^2 \sum_{k=1}^n (\basisfuncexp_k(\y_\indmeas) )^2
\right)
+
\dfrac{\M(\M-1)}{\M^2}   
\sum_{k=1}^n
\left(
\E \left(  
  \basisfuncexp_k w \err 
\right)
\right)^2
\\
& = 
\dfrac{n}{\M} \int_\Gamma  \err^2 \, d\mu + \dfrac{\M(\M-1)}{\M^2}   
\sum_{k=1}^n \left( 
\int_\Gamma \basisfuncexp_k \err \, d\mu 
\right)^2
\\
&= 
\dfrac{n}{\M} (e_2(\fexp))^2, 
\end{align*}
where 
at the last but one step 
we have used the definition of $w$, 
and at the last step we have used the orthogonality of $\err$ to $\basisfuncexp_k$ for all $k=1,\ldots,n$.
\end{proof}

\begin{proof}[Proof of Theorem~\ref{thm:estimates_quadrature_asy}]
Denote with $\mathbb{E}_{\widetilde{y}_1,\ldots,\widetilde{y}_{m}}$
the expectation over the random samples $\widetilde{y}_1,\ldots, \widetilde{y}_{m}$.
For given $y_1,\ldots,y_m$, 
using the mutual independence between $\widetilde{y}_1,\ldots,\widetilde{y}_\M$ and $\y_1,\ldots,\y_\M$  it holds that 
\begin{equation}
\mathbb{E}_{\widetilde{y}_1,\ldots,\widetilde{y}_{m}} \left(  \left| I(\fexp - \widetilde{\fexp}) - \dfrac{1}{m} \sum_{i=1}^{m} (\fexp - \widetilde{\fexp}) (\widetilde{y}_i) \right|^2   \right) = \dfrac{1}{m} \textrm{Var}_{\widetilde{y} \sim \mu}(\fexp - \widetilde{\fexp}),  
\label{eq:bound_exp_var}
\end{equation}
where the conditional variance on the right-hand side is defined as 
$$
\textrm{Var}_{\widetilde{y}\sim \mu}(\fexp(\widetilde{y}) - \widetilde{\fexp}(\widetilde{y})) := \mathbb{E}_{\widetilde{y} \sim \mu} 
\left( 
\left| \fexp(\widetilde{y})-\widetilde{\fexp}(\widetilde{y})  - 
\mathbb{E}_{
\widetilde{y} \sim \mu
}(\fexp(\widetilde{y}) - \widetilde{\fexp}(\widetilde{y}) )   
\right|^2 \right),
$$
using the following conditional expectation for the given $y_1,\ldots,y_m$: 
$$
\mathbb{E}_{\widetilde{y}\sim \mu}(\fexp(\widetilde{y}) - \widetilde{\fexp}(\widetilde{y})  ) := \int_\Gamma ( \fexp(\widetilde{y})-\widetilde{\fexp}(\widetilde{y}) ) \, d\mu(\widetilde{y}). 
$$
For any given $y_1,\ldots,y_m$, an upper bound for the variance is  
\begin{equation}
\textrm{Var}_{\widetilde{y} \sim \mu}(\fexp(\widetilde{y}) - \widetilde{\fexp}(\widetilde{y})) \leq 
 \int_\Gamma (\fexp(\widetilde{y}) - \widetilde{\fexp}(\widetilde{y}))^2 \, d\mu(\widetilde{y})
=
\| \phi - \widetilde{\phi}\|^2.  
\label{eq:upboundvar}
\end{equation}
Using the law of total expectation, 
\eqref{eq:bound_exp_var},
the upper bound 
\eqref{eq:upboundvar} 
and \eqref{eq:expect_estimate} 
we have that 
\begin{align*}
\mathbb{E}_{\widetilde{y}_1,\ldots,\widetilde{y}_{m} 
\atop 
y_1,\ldots,y_m 
} 
\left(  \left| I(\fexp)  - \widehat{I}_{2\M}(\fexp)
\right|^2   \right) 
= &
\mathbb{E}_{\widetilde{y}_1,\ldots,\widetilde{y}_{m} 
\atop 
y_1,\ldots,y_m 
} 
\left(  \left| I(\fexp - \widetilde{\fexp}) - \dfrac{1}{m} \sum_{i=1}^{m} (\fexp - \widetilde{\fexp}) (\widetilde{y}_i) \right|^2   \right) 
\\
= &
\mathbb{E}_{y_1,\ldots,y_m } \left( 
\mathbb{E}_{ \widetilde{y}_1,\ldots,\widetilde{y}_{m}    }
\left(  \left| I(\fexp - \widetilde{\fexp}) - \dfrac{1}{m} \sum_{i=1}^{m} (\fexp - \widetilde{\fexp}) (\widetilde{y}_i) \right|^2   \right) 
\right) 
\\
\leq & 
\dfrac{1}{m} \mathbb{E}_{y_1,\ldots,y_m } \left( \| \fexp - \widetilde{\fexp} \|^2
\right)
\\
\leq & 
\dfrac{1}{m}
\left(
\left( 1+\varepsilon(m) \right) \min_{v\in V_n} \| \fexp - v \|^2 + 2 \| \fexp \|^2 m^{-r}
\right). 
\end{align*}
\end{proof}

\begin{proof}[Proof of Theorem~\ref{thm:positive_weights}]
Proof of i). 
For any $\indmeas=1,\ldots,\M$, using the sub-multiplicative property of the operator norm we obtain that 
\begin{align}
\left\vert 
\quadweights_\indmeas
- \dfrac{w(\y_\indmeas)}{\M} \right\vert  
= 
\left\vert \dfrac{ \sqrt{ w(\y_\indmeas)} }{\M} \design_\indmeas \left( \gramian^{-1} - \cI \right) e_1 \right\vert  
\leq 
\dfrac{\sqrt{w(\y_\indmeas)}}{\M} 
\Vert \design_\indmeas \Vert_{\ell_2}
\ 
\vvvert \gramian^{-1} - \cI \vvvert
\ 
\Vert e_1 \Vert_{\ell_2}
.
\label{eq:each_w_i}
\end{align}
Now we estimate each term in the right-hand side of \eqref{eq:each_w_i}. 
For the second term, the definitions of $\design$  in \eqref{eq:def_D} and $w$ in \eqref{eq:weight_function} ensure that for any $\indmeas=1,\ldots,\M$ it holds 
\[
\Vert \design_\indmeas \Vert_{\ell_2}^2 = w(y_i) \sum_{j=1}^{n} \basisfuncexp_j(\y_\indmeas)^2 
=n. 
\]
For the third term, using \eqref{eq:statement_inverse} we have that, under condition \eqref{eq:condition_points},  
\[
\vvvert \gramian^{-1} - \cI \vvvert
\leq \dfrac{\delta}{1-\delta},
\]
with probability at least $1-2\M^{-\param}$. 
For the fourth term $\Vert e_1 \Vert_{\ell_2}=1$. 
We now observe that the restriction of $\delta$ to the interval   
\begin{equation}
\label{eq:condition_delta_positivity}
0 < \delta 
\leq 
\dfrac{ \sqrt{ w_\textrm{min} } }{2 \sqrt{n 
}  + \sqrt{ w_\textrm{min} } }, 
\quad n \geq 1,
\end{equation}
and strict monotonicity of $\delta \mapsto \delta (1-\delta)^{-1}$ on 
such an interval 
ensure that the left-hand side in 
\eqref{eq:each_w_i}
 satisfies the following upper bound,  uniformly for all $\indmeas=1,\ldots,\M$: 
\begin{equation}
\left|  
\quadweights_\indmeas
- \dfrac{w(\y_\indmeas)}{\M}
\right|
\leq 
\dfrac{\sqrt{w_{\textrm{min}}}}{\M}
\dfrac{\delta \sqrt{n} }{1-\delta} \leq \dfrac{ w_\textrm{min} }{2\M}.  
\label{eq:bound_unif_int_pesi}
\end{equation}

Since $w_{\textrm{min}}\leq 1$,
choosing 
\begin{equation}
\delta=\dfrac{ 
\sqrt{
w_\textrm{min}
} 
}{
3\sqrt{n 
} 
}
\leq 
\dfrac{ \sqrt{ w_\textrm{min} } }{2 \sqrt{n 
}  + \sqrt{ w_\textrm{min} } }, \quad n \geq 1,
\label{eq:choice_delta_pos_w} 
\end{equation}
and thanks to \eqref{eq:lower_bound_zeta_delta}, we can enforce condition \eqref{eq:condition_points} as
\begin{equation}
\label{eq:enforcing_cond}
\dfrac{\M}{\ln \M} 
\geq
\dfrac{3(1 + \param) n^2 
}{
(4\ln(4/3)-1)
w_{\textrm{min}}
}
= 
\dfrac{ (1 + \param) n }{ 
3(4\ln(4/3)-1)
\delta^2 } 
\geq
\dfrac{ (1 + \param ) n }{ \cdelta(\delta) }, 
\end{equation}
and obtain condition \eqref{eq:condition_points_pw}.
Condition \eqref{eq:condition_points_pw} ensures that \eqref{eq:bound_unif_int_pesi} holds with probability at least $1-2\M^{-\param}$ and simultaneously for all $\indmeas=1,\ldots,\M$, that is the claim \eqref{eq:thesis_positive_weights_th}.

Proof of ii). 
From \eqref{eq:enforcing_cond}, condition \eqref{eq:condition_points_pw} requires more points than \eqref{eq:condition_points}. Hence any cubature formula whose nodes are drawn from \eqref{eq:prob_meas_aux} and satisfy \eqref{eq:condition_points_pw}, 
yields an integration error which obeys to the convergence estimates in Theorem~\ref{thm:estimates_quadrature} but with $\delta$ chosen as in \eqref{eq:choice_delta_pos_w}.
Using the upper bounds \eqref{eq:upper_bound_w} and \eqref{eq:upper_bound_zeta_delta} one obtains \eqref{eq:upper_bound_eps_theo_wmin} in \eqref{eq:expectation_estimate}. 
Since $\delta \leq \frac13$ for any $n\geq 1$, \eqref{eq:expectation_estimate_L1_alt} holds true with an additional factor $3/2$ that multiplies $(1-\delta)^{-1}$. 
\end{proof}

\begin{remark}
\label{sec:rem_weights_sec}
With a similar proof as for \eqref{eq:thesis_positive_weights_th} but using $\| G^{-1} \|\leq (1-\delta)^{-1}$, 
under condition 
\eqref{eq:condition_points} 
for any $\delta \in (0,1)$ 
it holds that  
\begin{equation}
\Pr\left( 
\bigcap_{i=1}^\M
\left\{
|\widetilde{\alpha}_i | \leq 
  \dfrac{(n-1)\sqrt{n}}{   (1-\delta) } 
\dfrac{ 
\sqrt{
w_{\textrm{min}}
}
}{m}
\min\left( \sqrt{1+\delta},\dfrac{3\delta}{2}  \right)
\right\}
\right) 
> 1 -2 \M^{-r}, 
\label{eq:result_prob_weights_add}
\end{equation}
because from the polarization identity and $\psi_1\equiv 1$ we have for any $j>1$ that   
\begin{align*}
 \langle \psi_j, \psi_1 \rangle_\M & = \dfrac{ \| \psi_j + \psi_1 \|_\M^2 }{2} - \dfrac{ \| \psi_j  \|^2_\M }{2} - \dfrac{1}{2} \\
& \leq  
\begin{cases}
\| \psi_j\|_\M \leq \sqrt{1+\delta}, \\
\dfrac{1+\delta}{2} \| \psi_j + 1 \|^2 - \dfrac{ \|  \psi_j \|^2_\M}{2} -\dfrac{1}{2} 
= 
\dfrac{1+\delta}{2} (\| \psi_j  \|^2 +1 )  - \dfrac{ \|  \psi_j \|^2_\M}{2} -\dfrac{1}{2} \leq \dfrac{3\delta}{2},
\end{cases}
\end{align*}
and therefore 
$$
\Pr\left( 
\bigcap_{j=1}^n
\left\{
| \langle \psi_j, \psi_1 \rangle_\M |  \leq \min\left( \sqrt{1+\delta},\dfrac{3\delta}{2}  \right)  
\right\}
\right) 
> 1 -2 \M^{-r}.
$$
\indent
In \eqref{eq:result_prob_weights_add}, 
choosing $\delta$ 
depending on $n$ 
and proceeding similarly as in the proof of 
\eqref{eq:thesis_positive_weights_th}, 
one can 
obtain conditions on $\M$ ensuring that with large probability  
$\sum_{i=1}^\M |\widetilde{\alpha}_i | \leq C  \log \M $ 
 for some constant $C$.  
\end{remark}

\begin{lemma}
\begin{equation}
\cdelta(\delta) 
\leq \dfrac{\delta^2}{2}, \qquad \delta \in [0,1].
\label{eq:upper_bound_zeta_delta}
\end{equation}
\end{lemma}
\begin{proof}
Define $U(\delta):= \delta^2/2$ and $D(\delta):=U(\delta)-\cdelta(\delta)$. The function $D(\delta)$ is continuously differentiable at any $\delta \geq 0$, and strictly increasing since $dD(\delta)/d\delta = \delta - \ln(1+\delta)>0$ for any $\delta > 0$. 
Since $D(0)=0$, we have $D(\delta)\geq 0$ for any $\delta \in [0,1]$, and hence \eqref{eq:upper_bound_zeta_delta}.     
\end{proof}

\begin{lemma}
\begin{equation}
3(4\ln(4/3)-1)
\,
\delta^2
\leq 
\cdelta(\delta),  
\quad \delta \in \left[0,
\frac13
\right]. 
\label{eq:lower_bound_zeta_delta}
\end{equation}
\end{lemma}
\begin{proof}
Define $\omega:= 3(4\ln(4/3)-1) \approx 0.452$, 
$L(\delta):= \omega \delta^2$ and $D(\delta):=\cdelta(\delta)-L(\delta)$. 
The function $D(\delta)$ is twice continuously differentiable at any $\delta \geq 0$,   
 $dD(\delta)/d\delta = \ln(1+\delta) - 2 \omega\delta$, and $d^2D(\delta)/d\delta^2 = (1+\delta)^{-1} - 2 \omega$. 
Hence the function $D(\delta)$ is convex on $[0,(2\omega)^{-1}-1]$, 
concave on $[(2\omega)^{-1}-1,1/3]$, and 
since $D(0)=dD(0)/d\delta=D(1/3)=0$ it is also nonnegative over $[0,\frac13]$, that gives \eqref{eq:lower_bound_zeta_delta}.
\end{proof}

\section{Conclusions}
\label{sec:six}
\noindent
In any  domain $\Gamma \subseteq \mathbb{R}^\s$ with any dimension $\s\in\mathbb{N}$, we have constructed randomized cubature formulae that are stable and exact on a given space $V_n$ on $\Gamma$ with dimension $n$, under the assumption that an $L^2$ orthonormal basis of  $V_n$ is available in explicit form. 
In Theorem~\ref{thm:estimates_quadrature} 
we have proven that the integration error of these cubature formulae satisfies convergence estimates in probability \eqref{eq:probabilistic_estimate} and in expectation \eqref{eq:expectation_estimate}--\eqref{eq:expectation_estimate_L1_alt}, under condition \eqref{eq:condition_points} on the required number of nodes. 
Such a condition imposes a number of nodes only linearly proportional to $n$, up to a logarithmic term, thus approaching the number of nodes in Tchakaloff's theorem (see Theorem~\ref{thm:tchakaloff}), in the same general setting of arbitrary domain $\Gamma$ and arbitrary dimension $\s$.  
If the number of nodes satisfies the more demanding condition \eqref{eq:condition_points_pw}, where $\M$ is at least quadratically proportional to $n$ up to a logarithmic term, then the proposed randomized cubature formulae have strictly positive weights with high probability.  
Both conditions \eqref{eq:condition_points} and \eqref{eq:condition_points_pw} are immune to the curse of dimensionality: the required number of nodes only depends on $n$, and does not depend on the dimension $\s$.   
The rate of convergence with respect to $\M$ for the error in \eqref{eq:expectation_estimate_L1_alt} catches up with the convergence rate $\M^{-1/2}$ of Monte Carlo, but the multiplicative constant in \eqref{eq:expectation_estimate_L1_alt} can be much smaller thanks to the additional decay of the best approximation error in $V_n$ of $\fexp$. As a consequence, the proposed randomized cubatures provably outperform Monte Carlo whenever the best approximation error in $V_n$ of $\fexp$ decays faster than $n^{-1/2}$.

As a further contribution we have constructed also a cubature formula that is asymptotically optimal, but with error bounds that can be larger in the preasymptotic regime, 
see Theorem~\ref{thm:estimates_quadrature_asy} and Remark~\ref{conv_rates_remark}.

A point that has not been addressed in the present article is the choice of the space $V_n$. Such a choice depends indeed on the function $\fexp$, or on its smoothness class. 
In many applications, for example in the analysis of partial differential equations with parametric or stochastic data, a priori analyses provide good approximation spaces $V_n$ with proven convergence rates $n^{-s}$ with $s>0$. Whenever this is not the case, one can resort to an adaptive construction of the approximation space, see Remark~\ref{thm:remark_adaptive}.

The results on randomized cubatures in this article have been presented using always $\M$ identically distributed random samples from $\sigma$    
for the construction of the weighted least-squares estimator of the integrand function, for both cubatures in Theorem~\ref{thm:estimates_quadrature} and Theorem~\ref{thm:estimates_quadrature_asy}. 
The cubature in Theorem~\ref{thm:estimates_quadrature_asy} uses in addition $\M$ random samples from $\mu$, but not for the construction
of the weighted least-squares estimator.  
The whole analysis in this paper applies tout court to other types of (nonidentically distributed) random samples 
from other distributions than $\sigma$, \emph{e.g.}~the distribution used in \cite[Theorem 2]{M2018},  
and extends to the adaptive setting by exploiting recent advances on adaptive weighted least-squares estimators for approximating the integrand function.

\bibliographystyle{plain}

\begin{thebibliography}{99}
\bibitem{BT2006} C.Bayer, J.Teichmann: {\it The proof of Tchakaloff's theorem}, Proc. Amer. Math. Soc., 134:3035--3040, 2006. 
\bibitem{BG2004a} H.Bungartz, M.Griebel: {\it Sparse grids}, Acta Numer., 13:147--269, 2004.  
%
\bibitem{Cools} R.Cools: {\it Constructing cubature formulae: the science behind the art}, Acta Numer., 1--54, 1997.
%
%
%
\bibitem{CCMNT2013} A.Chkifa, A.Cohen, G.Migliorati, F.Nobile, R.Tempone: {\it Discrete least squares polynomial approximation with random evaluations - application to parametric and stochastic elliptic PDEs}, ESAIM Math. Model. Numer. Anal., 49(3):815--837, 2015. 
\bibitem{Christoffel} E.B.Christoffel: {\it Sur une classe particuli\`ere de fonctions enti\`eres et de fractions continues}, Ann.Mat.Pura Appl. 2:1--10, 1877. 
%
%
%
%

\bibitem{CM2016} A.Cohen, G.Migliorati: {\it Optimal weighted least-squares methods}, SMAI-JCM, 3:181--203, 2017.

\bibitem{DR84} P.Davis, P.Rabinowitz: {\it Methods of numerical integration}. Second edition,  Academic Press, 1984.
\bibitem{DKS2013a} J.Dick, F.Y.Kuo, I.H.Sloan: {\it High-dimensional integration: the quasi-{M}onte {C}arlo way}, Acta Numer., 22:133--288, 2013. 
%
\tc{\bibitem{EZ1960} S.M.Ermakov, V.G.Zolotukhin: {\it Polynomial approximations and the Monte Carlo method}, Theor. Prob. Appl. 5:428--431, 1960.}%
%
\tc{\bibitem{E1975} S.M.Ermakov: {\it Die Monte-Carlo-Methode und verwandte Fragen}, Oldenbourg Verlag, Mu\"nchen, 1975.}%
%
\bibitem{Gauss} C.F.Gauss: {\it Methodus nova integralium valores per approximationem inveniendi}, Comment.Soc.Regiae Sci.Gottingensis Recentiores, 3:163--196, 1816. 
\bibitem{GenzKeister} A.Genz, B.D.Keister: {\it Fully symmetric interpolatory rules for multiple integrals over infinite regions with Gaussian weight}, J.Comp.Appl.Math., 71:299--309, 1996. 


\bibitem{H1994} S.Heinrich: {\it Random approximation in numerical analysis}, in Proc. of the Functional Analysis Conf., Essen 1991, Lecture Notes in Pure and Applied Math. vol 150, Marcel Dekker, New York, pp.123--171. 

\bibitem{KUP} L.Kammerer, T.Ullrich, T.Volkmer: 
{\it Worst-case recovery guarantees for least squares approximation using random samples}, arXiv 1911.10111, November 2019.  

\bibitem{MS1967} J.Mc Namee, F.Stenger: {\it Construction of fully symmetric numerical integration formulas}, Numer.Math. 10:327--344, 1967. 

\bibitem{MNST2014} G.Migliorati, F.Nobile, E.von Schwerin, R.Tempone: {\it Analysis of discrete $L^2$ projection on polynomial spaces with random evaluations}, Found.~Comput.~Math., 14:419--456, 2014. 
%

\bibitem{M2018} G.Migliorati: {\it Adaptive approximation by optimal weighted least squares methods}, SIAM J.Numer.Anal., 57(5):2217--2245, 2019.

%
\bibitem{M2015} G.Migliorati: {\it Multivariate Markov-type and Nikolskii-type inequalities for polynomials associated with downward closed multi-index sets}, J. Approx. Theory, 189:137--159, 2015. 
%
%
%
\bibitem{Moller1979} H.M.M\"{o}ller: {\it Lower bounds for the number of nodes in cubature formulae}, Numerische Integration, 45:221--230, 1979. 
%
\tc{\bibitem{N88} E.Novak: {\it Deterministic and stochastic error bounds in numerical analysis}, Springer, 1988. }
%

\bibitem{N2015}  E.Novak: {\it Some Results on the Complexity of Numerical Integration}, In: Cools R., Nuyens D. (eds) Monte Carlo and Quasi-Monte Carlo Methods. Springer Proceedings in Mathematics \& Statistics, vol 163. Springer, Cham., 2016.

\bibitem{NR99} E.Novak, K.Ritter: {\it Simple cubature formulas with high polynomial exactness}, Constr.Approx. 15:499-522, 1999. 

\bibitem{Ott} J.Oettershagen: {\it Construction of optimal cubature algorithms with applications to econometrics and uncertainty quantification}, Phd thesis, 2017. 

\bibitem{RC} C.P.Robert, G.Casella: {\it Monte Carlo statistical methods}, Springer, 2013.  
\bibitem{P1997} M.Putinar: {\it A note on Tchakaloff's theorem}, Proceedings of the American Mathematical Society, 125(8):2409--2414, 1997.
\bibitem{RB} E.K.Ryu, S.P.Boyd: {\it Extensions of Gauss quadrature via linear programming}, Found. Comput. Math. 15:953--971, 2014.
\bibitem{Smolyak1963} S.Smolyak: {\it Quadrature and interpolation formulas for tensor products of certain classes of functions}, Dokl.~Akad.~Nauk SSSR, 4:240--243, 1963. 
\bibitem{Stroud1960} A.H.Stroud: {\it Quadrature methods for functions of more than one variable}, New York Acad.Sci. 86:776--791, 1960. 
\bibitem{Stroud1971} A.H.Stroud: {\it Approximate Calculation of Multiple Integrals}, Prentice-Hall, 1971. 
\bibitem{T1957} V.Tchakaloff: {\it Formules de cubature m\'echaniques \`a coefficients non n\'egatifs}, Bull.~Sci.~Math., 81:123--134, 1957. 
\bibitem{Tropp2015} J.A.Tropp: {\it The expected norm of a sum of independent random matrices: an elementary approach}, High-Dimensional Probability VII: The Cargese Volume. Eds. C. Houdre, D. M. Mason, P. Reynaud-Bouret, and J. Rosinski., Birkhaeuser, 2016. 



\end{thebibliography}

\end{document}